\documentclass[letterpaper,11pt]{article}

\newcommand*{\affaddr}[1]{#1} % No op here. Customize it for different styles.

\newcommand*{\email}[1]{\texttt{#1}}
\usepackage[numbers]{natbib}
\usepackage{hyperref} 
\usepackage{amsmath}
\usepackage{amssymb} %math symbols
\usepackage{amsthm}
\usepackage{geometry} % control margins
\usepackage{enumitem}
\usepackage{graphicx}
\usepackage{float}  % for floating objects e.g. figures, tables, ...
\usepackage{subcaption} % for subfigures, etc.
\usepackage{color}  % text coloring
\usepackage{algorithm}
\usepackage{algorithmic}
\usepackage[normalem]{ulem} % for revising
\usepackage[bottom]{footmisc}
\usepackage[export]{adjustbox}
\usepackage{upgreek}
\usepackage{todonotes}
\usepackage{ulem}

\providecommand{\keywords}[1]
{
  \small	
  \textbf{\textit{Keywords---}} #1
}

\DeclareMathOperator*{\argmin}{argmin}

\def\fprod#1{\left\langle#1\right\rangle}

\newcommand\numberthis{\addtocounter{equation}{1}\tag{\theequation}}

\newtheorem{theorem}{Theorem}[section]
\newtheorem{corollary}{Corollary}[section]
\newtheorem{lemma}{Lemma}[section]
\newtheorem{remark}{Remark}
\newtheorem{definition}{Definition}[section]
\usepackage{natbib}
 \bibpunct[, ]{(}{)}{,}{a}{}{,}%
\def\bmu{\mathbf{\mu}}
\def\bmu{\boldsymbol\mu}

\def\ba{\mathbf{a}}

\def\bg{\mathbf{g}}
\def\bh{\mathbf{h}}

\def\bu{\mathbf{u}}
\def\bv{\mathbf{v}}
\def\bw{\mathbf{w}}
\def\bx{\mathbf{x}}  %{\mbox{\boldmath $\lambda$}}
\def\by{\mathbf{y}}
\def\bz{\mathbf{z}}

\def\bH{\mathbf{H}}
\def\bI{\mathbf{I}}

\def\bU{\mathbf{U}}

\def\m{{\boldsymbol{\mu}}}

\def\cE{\mathcal{E}}

\def\cL{\mathcal{L}}
\def\cM{\mathcal{M}}

\def\cS{\mathcal{S}}

\def\cU{\mathcal{U}}

\def\mE{\mathbb{E}}

\def\mN{\mathbb{N}}

\def\mR{\mathbb{R}}

\def\smskip{\smallskip}

\def\texitem#1{\par\smskip\noindent\hangindent 25pt
               \hbox to 25pt {\hss #1 ~}\ignorespaces}

\def\norm#1{\|#1\|}

\newcommand{\BEAS}{\begin{eqnarray*}}
\newcommand{\EEAS}{\end{eqnarray*}}
\newcommand{\BEA}{\begin{eqnarray}}
\newcommand{\EEA}{\end{eqnarray}}
\newcommand{\BEQ}{\begin{eqnarray}}
\newcommand{\EEQ}{\end{eqnarray}}
\newcommand{\BIT}{\begin{itemize}}
\newcommand{\EIT}{\end{itemize}}
\newcommand{\BNUM}{\begin{enumerate}}
\newcommand{\ENUM}{\end{enumerate}}

\newcommand{\BA}{\begin{array}}
\newcommand{\EA}{\end{array}}

\newif\ifpagenumbering
\pagenumberingtrue

\pagenumberingfalse

\newsavebox{\theorembox}
\newsavebox{\lemmabox}
\newsavebox{\defnbox}
\newsavebox{\assbox}
\savebox{\theorembox}{\noindent\bf Theorem}
\savebox{\lemmabox}{\noindent\bf Lemma}
\savebox{\defnbox}{\noindent\bf Definition}
\savebox{\assbox}{\noindent\bf Assumption}
\newtheorem{assumption}{\usebox{\assbox}}
\begin{document}
% Don't want date printed
\date{\today}
% Make title large and bold
\title{\vspace{-1.5cm}\Large\bfseries Riemannian Stochastic Variance-Reduced Cubic Regularized Newton Method for Submanifold Optimization}

% Target typesetting:
%
% Author A, Author B, Author C, Author D and Author E
%        A,B,C Department of Computer Science
%       D,E Department of Mechanical Engineering
%          Email A,B,C,D,E @university.edu
%                  Latex University

\author{%
Dewei Zhang and Sam Davanloo Tajbakhsh\thanks{Corresponding author}
%, Author C\affmark[1], Author D\affmark[2], and Author E\affmark[2]\thanks{333}
\vspace{0.2cm}\\
%\affaddr{\affmark[1]Department of Integrated Systems Engineering}\\
\email{\{zhang.8705,davanloo.1\}@osu.edu \vspace{0.4cm}}\\
\affaddr{The Ohio State University}%
}

\maketitle

\vspace{-0.6cm}
\begin{abstract}
We propose a stochastic variance-reduced cubic regularized Newton algorithm to optimize the finite-sum problem over a Riemannian submanifold of the Euclidean space. The proposed algorithm requires a full gradient and Hessian update at the beginning of each epoch while it performs stochastic variance-reduced updates in the iterations within each epoch. The iteration complexity of $O(\epsilon^{-3/2})$ to obtain an $(\epsilon,\sqrt{\epsilon})$-second-order stationary point, i.e., a point with the Riemannian gradient norm upper bounded by $\epsilon$ and minimum eigenvalue of Riemannian Hessian lower bounded by $-\sqrt{\epsilon}$, is established when the manifold is embedded in the Euclidean space. Furthermore, the paper proposes a computationally more appealing modification of the algorithm which only requires an \emph{inexact} solution of the cubic regularized Newton subproblem with the same iteration complexity. The proposed algorithm is evaluated and compared with three other Riemannian second-order methods over two numerical studies on estimating the inverse scale matrix of the multivariate t-distribution on the manifold of symmetric positive definite matrices and estimating the parameter of a linear classifier on the Sphere manifold.
\end{abstract}

%\KEYWORDS{deterministic inventory theory; infinite linear programming duality;
%  existence of optimal policies; semi-Markov decision process; cyclic schedule}


%
%\smartqed
%%This command right justifies \qed throughout the paper.
%\usepackage{graphicx}
%%This package is used to insert figures.
%
%%%%%%%%%%%%%%%%%% SDT %%%%%%%%%%%%%%%
%\usepackage{hyperref} 
%\usepackage{amsmath}
%\usepackage{amssymb} %math symbols
%%\usepackage{amsthm}
%\usepackage{geometry} % control margins
%\usepackage{enumitem}
%\usepackage{graphicx}
%\usepackage{float}  % for floating objects e.g. figures, tables, ...
%\usepackage{subcaption} % for subfigures, etc.
%\usepackage{color}  % text coloring
%\usepackage{algorithm}
%\usepackage{algorithmic}
%\usepackage[normalem]{ulem} % for revising
%\usepackage[bottom]{footmisc}
%\usepackage[export]{adjustbox}
%\usepackage{upgreek}
%\usepackage{todonotes}
%\usepackage{ulem}
%\usepackage{comment}
%\usepackage{placeins}
%
%
%
%
%\input defs.tex
%
%
%\DeclareMathOperator*{\supp}{supp}
%\DeclareMathOperator*{\argmin}{argmin}
%\DeclareMathOperator*{\argmax}{argmax}
%\def\fprod#1{\left\langle#1\right\rangle}
%\def\prox#1{\mathbf{prox}_{#1}}
%\def\HessU{\tilde{\bU}^s_t}
%\newcommand\numberthis{\addtocounter{equation}{1}\tag{\theequation}}
%%%%%%%%%%%%%%%%%% SDT %%%%%%%%%%%%%%%
%
%
%
%
%
%
%

\vspace{-0.2cm}
\section{Introduction}\label{sec:intro}
We study the optimization of the finite-sum problem over a Riemannian manifold $\cM$ embedded in a Euclidean space $\cE$ as 
\begin{equation}\label{eq:main1}
\min_{\bx\in \cM \subseteq \cE} F(\bx)= \frac{1}{N} \sum^N_{i=1} f_i(\bx),
\end{equation}
where $N$ is a (possibly very large) positive integer. Manifold optimization has a range of applications in machine learning, statistics, control and robotics, e.g., in deep learning, low-rank matrix completion, sparse or nonnegative principal component analysis, or solving large-scale semidefinite programs
%approximation models for integer programming
-- see \cite{hu2020brief,absil2019collection} and the references therein.
The finite-sum structure of the objective function in problem \eqref{eq:main1} specifically finds applications in machine learning and statistics for parameter estimation, and addition of the manifold constraint could have problem-specific, computational, or other reasons. Below, we provide two motivational examples for problem~\eqref{eq:main1}.

\vspace{-0.2cm}
\paragraph{Example~1 (Parameter estimation of the multivariate Student's t-distribution)}~As an important member of the family of elliptical distributions~\cite{domino2018selected}, the multivariate t-distribution has numerous applications in mathematical finance, survival analysis, biology, etc.~\cite{kotz2004multivariate}. For instance in mathematical finance~\cite{szego2002measures,de2004measuring,krzanowski1994multivariate}, the \emph{Student's t copula} $C_{\nu}:[0,1]^p\to[0,1]$ defined as
$C_{\nu}(\bu)= T_{\nu,\Sigma,\m}(T^{-1}_\nu(u_1),...,T^{-1}_\nu(u_n))$,
where $T_{\nu,\Sigma,\m}(\bx)$ is the multivariate Student's t cumulative density function (CDF) and $T_\nu^{-1}(\cdot)$ is the inverse of the marginal univariate Student's t CDF is used  to model or sample from multivariate Student's t-distribution~\cite{krzanowski1994multivariate}. 
%This copula is extensively used to model financial data~.
As one of the core tasks, the maximum likelihood parameter estimation of the multivariate Student's t-distribution requires solving
\begin{equation}\label{student_t}
\max_{\Sigma\in\cS^p_{++},\bmu\in\mR^p} \ \ \frac{1}{N}\sum_{i=1}^N \log(t_{\nu}(\bx_i;\m,\Sigma)),
\end{equation}
where $t_{\nu}(\bx;\m,\Sigma)$ denotes the probability density function of the multivariate t-distribution 
%$$t_{\nu}(\bx;\m,\Sigma)=\frac{\Gamma[(\nu+p)/2]}{\Gamma(\nu/2)\nu^{p/2}\pi^{p/2}|\bSigma|^{1/2}}[1+\frac{1}{\nu}(\bx-\bmu)^T\bSigma^{-1}(\bx-\bmu)]^{-(\nu+p)/2},$$
and $\nu\in\mN_+$ is the degrees of freedom which is generally predetermined. Since the scale matrix $\Sigma$ should belong to the manifold of positive-definite matrices, problem \eqref{student_t} is an instance of problem \eqref{eq:main1}.

\vspace{-0.2cm}
\paragraph{Example~2 (Efficient training of deep neural networks)}~Training deep neural networks could be ``notoriously difficult" when the singular values of the hidden-to-hidden weight matrices deviates from one~\cite{arjovsky2016unitary}. In such cases optimization becomes difficult due to the \emph{vanishing} or \emph{exploding} gradient~\cite{arjovsky2016unitary,wisdom2016full}. This challenge can be circumvented, if the weight matrices are unitary with singular values equal to one. This can be achieved by requiring hidden-to-hidden weight matrices to belong to Stiefel  Manifold $St(p,n)\triangleq\{X\in\mR^{n\times p}:\ X^\top X=I_p\}$~\cite{absil2009optimization,boumal2020introduction} and training the model using a Riemannian optimization algorithm. The underlying optimization problem is an instance of \eqref{eq:main1} over the cartesian product of Stiefel manifolds. The orthonormality of the weight matrices improves the performance of deep neural networks~\cite{bansal2018can,li2020efficient}, reduces overfitting to improve generalization~\cite{cogswell2015reducing}, or stabilizes the distribution of activations over
layers~\cite{huang2018orthogonal}. In \cite{sun2017svdnet} and \cite{xie2017all} two convolutional neural networks are trained with orthonormal weight matrices for phase retrieval and image classification.

\vspace{-0.2cm}
\subsection{Related Work} \label{sec:related_work}
Numerous algorithms for standard unconstrained optimization~\cite{ruszczynski2011nonlinear} have been generalized to Riemannian manifolds~\cite{absil2009optimization,udriste2013convex,boumal2020introduction}. Some notable first-order algorithms include
%Nelder-Mead method~\cite{dreisigmeyer2007direct}, 
gradient descent method~\cite{zhang2016first,boumal2019global}, conjugate gradient method~\cite{smith1994optimization,sato2015new}, stochastic gradient method~\cite{bonnabel2013stochastic,zhang2016riemannian,tripuraneni2018averaging}, accelerated methods~\cite{liu2017accelerated,ahn2020nesterov,zhang2018towards,criscitiello2020accelerated,alimisis2021momentum,alimisis2020continuous}, proximal gradient methods~\cite{ferreira2002proximal,bento2015proximal,bento2017iteration,de2016new,huang2021riemannian}. To guarantee convergence to a second-order stationary point, \cite{jin2019stochastic} investigates the Riemannian perturbed gradient descent that guarantees second-order stationarity without using second-order information with $\tilde{O}(1/\epsilon^2)$ iteration complexity. Other saddle-escape methods over manifolds were also studied in \cite{sun2019escaping} and \cite{criscitiello2019efficiently}.

In the context of second-order algorithms, the Newton method is extended to optimize over Riemannian manifolds in~\cite{luenberger1972gradient,gabay1982minimizing,smith1993geometric,smith1994optimization}. Specifically, \cite{smith1993geometric} and \cite{smith1994optimization} establish local quadratic convergence. Similar to the Newton method on Euclidean space, the Newton method for manifold optimization also suffers from two main drawbacks: first, it is possible that the Hessian matrix is degenerated at a point; second, it is possible that the iterates diverge, converge to a saddle point, or even a local maximum. In the Riemannian setting, the trust region method \cite{absil2004trust,absil2007trust,baker2008implicit,boumal2015riemannian,boumal2019global} and the cubic regularized Newton method \cite{zhang2018cubic} are extensions of their Euclidean counterparts to address these drawbacks. More specifically, \cite{boumal2019global} shows that the Riemannian trust region method obtains an $(\epsilon,\epsilon)$-second-order-stationary point (see Definition~\ref{def:eps_delta}) in $O(1/\epsilon^3)$ which matches its Euclidean counterpart~\cite{cartis2012complexity,cartis2014complexity}. Furthermore, the cubic regularized Newton method on manifolds \cite{zhang2018cubic} is shown to reach an $(\epsilon,\sqrt{\epsilon})$-second-order stationary point in $O(1/\epsilon^{1.5})$. Finally,~\cite{agarwal2018adaptive} extends the daptive cubic regularization method \cite{cartis2011adaptive} to Riemannian manifolds and establishes $O(1/\epsilon^{1.5})$ rate to obtain an $(\epsilon,\sqrt{\epsilon})$-second-order stationary point -- see also~\cite{qi2011numerical,hu2018adaptive}.

A common issue among second-order algorithms is their high computational cost to calculate the inverse of the Hessian operator. A Riemannian counterpart of the famous BFGS algorithm \cite{nocedal2006numerical} is proposed in \cite{ring2012optimization} which does not require calculating the inverse of the Hessian operator. 

To optimize functions with the finite-sum structure or those that are known through their approximate gradient and Hessian, inexact methods including first- and second-order Riemannian stochastic algorithms and their variance-reduced extensions are proposed in the literature.
As a generalization of~\cite{johnson2013accelerating}, the Riemannian stochastic variance-reduced gradient descent (SVRG) method was developed in \cite{zhang2016riemannian}. Furthermore, an extension of  Riemannian SVRG with computationally more efficient retraction and vector transport was developed in \cite{sato2019riemannian}. The paper also establishes global convergence properties of their method besides its local convergence rate. %In~\cite{tripuraneni2018averaging}, they adapted the Polyak-Ruppert iterate averaging into the Riemannian setting~\cite{ruppert1988efficient,polyak1990new}. 
The Riemannian version of the stochastic recursive gradient method~\cite{nguyen2017stochastic} is proposed in~\cite{kasai2018riemannian}.
\cite{kasai2018inexact} proposes Riemannian trust region algorithms with inexact gradient and Hessian that allows inexact solution of the subproblem. Furthermore, a Riemannian stochastic variance-reduce quasi-Newton method is proposed in \cite{kasai2017riemannian} - see also \cite{roychowdhury2017accelerated}. For a recent review of first- and second-order Riemannian optimization algorithms, we refer the reader to \cite{hosseini2020recent} and \cite{sato2021riemannian} (specifically, Section~6.1 on stochastic methods).

\subsection{Contributions}
The major contributions of this paper are as follows:
(i) Motivated by \cite{pmlr-v80-zhou18d} and \cite{kovalev2019stochastic} in the Euclidean setting, we propose a stochastic variance-reduced cubic regularized Newton method (R-SVRC algorithm) to optimize over Riemannian manifolds.
(ii) We carefully analyze the worst-case complexity of the proposed algorithm to find a point that satisfy the first- and second-order necessary optimality conditions (i.e., a second-order stationarity point) when the cubic-regularized Newton subproblem is solved \emph{exactly} - see Theorem~\ref{main1} and Corollary~\ref{corollary1}. (iii) We performed the analysis of a computationally more appealing version of the algorithm, that allows solving the cubic-regularized Newton subproblem \emph{inexactly}, and established the same worst-case complexity bound - see Theorem~\ref{main2} and Corollary~\ref{cor:inexact}. The assumptions for our analysis are explicitly discussed in Section~\ref{sec:con_analys}. Finally, the performance of the proposed algorithm is evaluated and compared over two applications: 1. Estimating the scale matrix of Student's t-distribution over the symmetric positive definite manifold, 2. Learning the parameter of a linear classifier over a Sphere manifold. The implementation of the proposed algorithm in MATLAB with exact and inexact subproblem solvers is provided at \url{https://github.com/samdavanloo/R-SVRC}. To the best of our knowledge, this work is the first \emph{stochastic} Newton method with cubic regularization on Riemannian manifold.

\subsection{Preliminaries and Notation}
A Riemannnian manifold $(\cM,g)$ is a real smooth manifold $\cM$ equipped with a Riemannain metric $g$. The metric $g$ induces an inner product structure in each tangent space $T_{\bx} \cM$ associated with point $\bx\in\cM$. We denote the inner product of $\bu, \bv\in T_{\bx}\cM$ as $\fprod{\bu,\bv}_\bx=g_\bx(\bu,\bv)$, and the norm of $\bu$ is defined as $\norm{\bu}=\sqrt{g_x(\bu,\bu)}$. Furthermore, the angle between $\bu$ and $\bv$ is $\arccos(\fprod{\bu,\bv}_{\bx}/(\norm{\bu}\norm{\bv}))$. Given a smooth real-valued function $f$ on a Riemannian manifold $\cM$, Riemannian gradient and Hessian of $f$ at $\bx$ are denoted by $\text{grad}f(\bx)$ and $\text{Hess}f(\bx)$ (also for simplicity by $H_\bx$). For a symmetric operator, e.g. the Riemannian Hessian $H$ at $\bx\in\cM$, the operator norm of $H$ is defined as $\norm{H}_{op}=\sup\{\norm{H\eta}:\eta\in T_{\bx}\cM, \norm{\eta}=1\}$. An operator on $T_{\bx} \cM$ is positive semidefinite $H\succeq 0$ if $\fprod{H[\eta],\eta}\geq 0$, for any $\eta\in T_x {\cM}$. A geodesic is a constant speed curve $\gamma: [0,1]\rightarrow\cM$ that is locally distance minimizing. An exponential map $\text{Exp}_x: T_x \cM\rightarrow\cM$ maps $\bv\in T_x\cM$ to $\by\in\cM$, such that there is a geodesic $\gamma$ with $\gamma(0)=\bx$, $\gamma(1)=\by$, and $\dot{\gamma}(0)=\bv$. For two points $\bx$, $\by\in\cM$, and $d(\bx,\by)<\text{inj}(\cM)$, there is a unique geodesic. The exponential map has an inverse $\text{Exp}_x^{-1}: \cM\rightarrow T_x\cM$ and the geodesic is the unique shortest path with $\norm{\text{Exp}_{\bx}^{-1}(\by)}=\norm{\text{Exp}_{\by}^{-1}(\bx)}$ the geodesic distance between $\bx$, $\by\in\cM$. Parallel transport $\Gamma_\bx^\by: T_\bx \cM\rightarrow T_\by\cM$ maps a vector $\bv\in T_\bx \cM$ to $\Gamma_\bx^\by\bv\in T_\by\cM$, while preserving norm, and roughly speaking ``direction". A tangent vector of a geodesic $\gamma$ remains tangent if parallel transported along $\gamma$. Parallel transport also preserves inner products, i.e. $\fprod{\bu,\bv}_\bx=\fprod{\Gamma(\gamma)^\by_\bx \bu,\Gamma(\gamma)^\by_\bx \bv}_\by$. We denote the orthogonal projection operator onto $T_\bx\cM$ by $P_{\bx}$.

Let $(\cM,g)$ be a connected Riemannian manifold (see e.g. \cite{absil2009optimization}) which carries the structure of a metric space whose distance function is the arc length of a minimizing path between two points.

\begin{definition}[Riemannian metric] An inner product on $T_{\bx}\cM$ is a bilinear, symmetric, positive definitive function $\fprod{\cdot,\cdot}_{\bx}: T_{\bx}\cM\times T_{\bx}\cM\rightarrow \mR$. It induces a norm for tangent vectors as $\norm{u}_{\bx}=\sqrt{\fprod{u,u}_{\bx}}$. The smoothly varying inner product is called the Riemannian metric, i.e., if $\bv$, $\bw$ are two smooth vector fields on $\cM$ then the function $\bx\mapsto \fprod{\bv(\bx),\bw(\bx)}_{\bx}$ is smooth from $\cM$ to $\mR$.
\end{definition}

\begin{remark}
	 The inner product of two elements $\xi_{\bx}$ and $\zeta_{\bx}$ of $T_{\bx}\cM$ are interchangeably denoted by $g(\xi_{\bx},\zeta_{\bx})=g_{\bx}(\xi_{\bx},\zeta_{\bx})=\fprod{\xi_{\bx},\zeta_{\bx}}=\fprod{\xi_{\bx},\zeta_{\bx}}_{\bx}$.
\end{remark}

\begin{definition}[Injectivity radius~\cite{boumal2020introduction}]\label{def:injec_radius}
The injectivity radius of a Riemannian manifold $\cM$ at a point $\bx$, denoted by $\text{inj}(\bx)$, is the supremum over radius $r>0$ such that $\text{Exp}_{\bx}$ is defined and is a diffeomorphism on the open ball 
$B(\bx,r)=\{v\in T_{\bx}\cM: \norm{v}_{\bx}< r\}$. By the inverse function theorem, $\text{inj}(\bx)>0$. Furthermore, the injectivity radius of a Riemannian manifold $\cM$, i.e., $\text{inj}(\cM)$, is the infimum of $\text{inj}(\bx)$ over $\bx\in\cM$~(\cite{boumal2020introduction}, Definition~10.14).
\end{definition}

Consider the ball $U\triangleq B(\bx,\text{inj}(\bx))\subseteq T_{\bx}\cM$ in the tangent space at $\bx$. Its image $\cU\triangleq\text{Exp}_{\bx}(U)$ is a neighborhood of $\bx$ in $\cM$. By definition, $\text{Exp}_{\bx}:U\rightarrow\cU$ is a diffeomorphism, with well-defined, smooth inverse $\text{Exp}^{-1}_{\bx}:\cU\rightarrow U$. With these choices of domains, $v=\text{Exp}^{-1}_{\bx}(\by)$ is the unique shortest tangent vector at $\bx$ such that $\text{Exp}_{\bx}(v)=\by$.

\begin{definition}[Riemannian distance] The Riemannian distance on a connected Riemannian manifold $(\cM,g)$ is 
\begin{equation}\label{distance}
    d: \cM \times \cM \rightarrow \mR: d(\bx,\by)\triangleq\inf_{\gamma\in\Upgamma} L(\gamma),
\end{equation}
where $L(\gamma)=\int_a^b{\sqrt{g_{\gamma(t)}(\dot{\gamma}(t),\dot{\gamma}(t))}dt}$ and $\Upgamma$ is the set of all curves in $\cM$ joining points $\bx$ and $\by$. Specifically, if $\norm{v}_{\bx}<\text{inj}(\bx)$ then $d(\bx,\text{Exp}_{\bx}(v))=\norm{v}_{\bx}$.
%See \cite{absil2009optimization}. 
%We present a formal prove for completeness. 
\end{definition}

\begin{definition}[Riemannian gradient] Given a smooth real-valued function $f$ on a Riemannian manifold $\cM$, the Riemannian gradient of $f$ at a point $\bx\in\cM$, denoted by $\text{grad} f(\bx)$, is defined as the unique element of $T_{\bx}\cM$ that satisfies
\vspace{-0.2cm}
\begin{equation}
	\fprod{\text{grad} f(\bx),\xi}_{\bx}=\text{D} f(\bx)[\xi], \ \forall\xi\in T_{\bx}\cM.
\end{equation}
\end{definition}

Specifically, when $\cM$ is a Riemmannian submanifold of the Euclidean space $\mR^{m\times n}$,
\begin{equation}
	\text{grad} f(\bx)=\text{P}_{\bx} \nabla f(\bx),
\end{equation}
where $\text{P}_{\bx}$ is the Euclidean projection onto $T_{\bx}\cM$ which is a nonexpansive linear transformation. 
\begin{definition}[Riemannian Hessian] Given a real-valued function $f$ on a Riemannian manifold $\cM$, the Riemannian Hessian of $f$ at a point $\bx\in\cM$ is the linear mapping $\text{Hess} f(\bx)$ from $T_{\bx} \cM$ onto itself defined as 
\begin{equation}
	\text{Hess} f(\bx)[\xi]=\nabla_\xi \text{grad} f
\end{equation}
for all $\xi\in\cM$, where $\nabla$ is the Riemannian connection on $\cM$~\cite{absil2009optimization}. 
\end{definition}
When $\cM$ is a Riemmannian submanifold of the Euclidean space $\mR^{m\times n}$, the Riemannian Hessian of $f$ is written as 
\begin{equation}\label{Hessian_with_D}
	\text{Hess} f(\bx)[\xi]=\text{P}_{\bx} (\text{D} \text{grad} f(\bx)[\xi]),
\end{equation}
i.e., the classical directional derivatives followed by an orthogonal projection. For more information, e.g., refer to Proposition 5.3.2 in \cite{absil2009optimization}.

\vspace{-0.2cm}
\section{Proposed Algorithm} \label{sec:alg}
The proposed Riemannian Stochastic Variance-Reduced Cubic Regularization (R-SVRC) method is presented in Algorithm~\ref{alg:svr_cubic}. This algorithm is indeed semi-stochastic, which requires calculation of full gradient and Hessian at the beginning of each epoch $s$, i.e., the outer loop in the algorithm. However, within each epoch, there are $T$ iterations of the inner loop, which require calculation of stochastic variance-reduced gradient and Hessian by sampling $|I_g|$ and $|I_h|$ components, respectively.

From the computational perspective, the major step of the algorithm is to solve the cubic-regularized Newton subproblem. Under Assumption~\ref{assumption:embedded}, the manifold is embedded in $\mR^{m\times n}$; hence, the tangent vectors in $T_{\bx}\cM$ are naturally represented by $m\times n$ matrices (see~\cite{absil2009optimization}). Therefore, current solvers for cubic regularized problem in Euclidean space can be adopted~\cite{manopt}. Solving this generally nonconvex subproblem is discussed in more details below \eqref{m_cubic}. For computational gain, the paper also considers solving the subproblem inexactly. As long as the inexact solution satisfies the conditions in Definition~\ref{def:inexact_sol}, our analysis guarantees the results of the exact case. 

%\blue{(delete this part)
%Finally, we note that in practice it is difficult to characterize the Hessian operators $\text{Hess} f$ and $\bU^s_t$. The difficulty is two-fold: First, it is nontrivial to find a basis system $\{\bb_i\}$ that spans $T_{\bx}\cM$ \footnote{But if such a basis system is found, then $\{\Gamma_{\bx}^{\by} b_i\}$ automatically becomes a basis for $T_{\by}\cM$ due to the isometric property of parallel transport $\Gamma$. This means that the manifold is locally Euclidean, and we can find a ``unifying" basis for different $T_{\bx} \cM$.}. Second, if such basis system exists, since $T_{\bx}\cM$ can have high dimension, we need to calculate many $\fprod{\bU^s_t b_i,b_j}$ which is computationally demanding. Fortunately, it is not necessary to characterize the operator matrix, and it is possible to efficiently evaluate $\bU^s_t[\eta]$ for any $\eta$ -- see.
%}

%\vspace{0.2cm}
\begin{algorithm}[H]
\caption{Riemannian Stochastic Variance-Reduced Cubic Regularization (R-SVRC)}
\label{alg:svr_cubic}
\begin{algorithmic}[1]
\REQUIRE batch size parameters $b_g$, $b_h$, cubic penalty parameters $\sigma$, number of epochs $S$, epoch length $T$, and a starting point $\bx_0$.
\STATE Set $\hat{\bx}^1=\bx_0$
%\INITIALIZATION as $\hat{\bx}^1=\bx_0$
\FOR{$s=1,...,S$} 
\STATE $\bx^s_0=\hat{\bx}^s$
\STATE $\bg^s=\text{grad} F(\hat{\bx}^s)=\frac{1}{N}\sum^N_{i=1}\text{grad} f_i(\hat{\bx}^s)$; $\bH^s=\text{Hess}F(\hat{\bx}^s)=\frac{1}{N}\sum^N_{i=1}\text{Hess}f_i(\hat{\bx}^s)$
\FOR{$t=0,...,T-1$}
\STATE Sample index set $I_g$, $I_h$, s.t. $|I_g|=b_g$, $|I_h|=b_h$
\STATE Compute $\hat{\eta}^s_t\in T_{\hat{\bx}^s}$, s.t. $\text{Exp}_{\hat{\bx}^s}(\hat{\eta}^s_t)=\bx^s_t$
\STATE \scriptsize{$\bv^s_t=\Upgamma_{\hat{\bx}^s}^{\bx_t^s}(\bg^s)
+\frac{1}{b_g}\big(\sum_{i_t\in I_g}\text{grad} f_{i_t}(\bx^s_t)-\Upgamma_{\hat{\bx}^s}^{\bx_t^s}(\sum_{i_t\in I_g}\text{grad}f_{i_t}(\hat{\bx}^s))\big)
-\Upgamma_{\hat{\bx}^s}^{\bx_t^s}\big(\frac{1}{b_g}\sum_{i_t\in I_g}\text{Hess}f_{i_t}(\hat{\bx}^s)-\bH^s\big)\hat{\eta}^s_t$}

\STATE $\bh^s_t=\argmin_{\bh\in T_{\bx}\cM}\fprod{\bv^s_t,\bh}+\frac{1}{2}\fprod{\bU^s_t\bh,\bh}+\frac{\sigma}{6}\|\bh\|^3$\label{algo:subproblem_cubic}, where \\
	\qquad	\scriptsize{$\bU^s_t=\Upgamma_{\hat{\bx}^s}^{\bx_t^s}\circ\bH^s\circ\Upgamma^{\hat{\bx}^s}_{\bx_t^s}+\frac{1}{b_h}\sum_{j_t\in I_h}\text{Hess}f_{j_t}(\bx^s_t)-\frac{1}{b_h}\Upgamma_{\hat{\bx}^s}^{\bx_t^s}\circ(\sum_{j_t\in I_h}\text{Hess}f_{j_t}(\hat{\bx}^s))\circ\Upgamma^{\hat{\bx}^s}_{\bx_t^s}$}
\STATE $\bx^s_{t+1}=\text{Exp}_{\bx^s_t}(\bh^s_t)$
\ENDFOR
\STATE $\hat{\bx}^{s+1}=\bx^{s+1}_T$
\ENDFOR
\RETURN $\bx_{out}=\bx^s_t$, where $s$, $t$ are uniformly at random chosen from $s\in [S]$ and $t\in [T]$
%\hspace{-0.8cm }\textbf{Output:} $(\b,\x)$
\end{algorithmic}
\end{algorithm}

\section{Lipschitzian Smoothness on Riemannian Manifolds}\label{sec:smoothness}
\begin{definition}[g-smoothness~\cite{da1998geodesic,ferreira2019gradient}] A differentiable function $f: \cM\rightarrow\mR$ is said to be geodesically $L_g$-smooth if its gradient is $L_g$-Lipschitz, i.e., for any $\bx$, $\by\in\cM$ with $d(\bx,\by)<\text{inj}(\cM)$, 
\begin{equation}\label{l1}
    \norm{\text{grad} f(\bx)-\Gamma^\bx_\by \text{grad} f(\by)}_{\bx}\leq L_g d(\bx,\by),
\end{equation}
where $\Gamma^\bx_\by$ is the parallel transport from $\by$ to $\bx$ following the unique minimizing geodesic connecting $\bx$ and $\by$.
\end{definition}
It can be proven that if $f$ is $L_g$-smooth, then for any $\bx$, $\by\in\cM$ with $d(\bx,\by)<\text{inj}(\cM)$, 
\begin{equation}\label{lc1}
    |f(\by)-(f(\bx)+\fprod{\text{Exp}^{-1}_\bx(\by),\text{grad} f(\bx)}_{\bx})|\leq \frac{L_g}{2} d^2(\bx,\by)
\end{equation}
-- see, e.g. \cite{bento2017iteration}, Lemma 2.1.
\begin{definition}[H-smoothness~\cite{agarwal2020adaptive}] A twice differentiable function $f:\cM\rightarrow\mR$ is said to be geodesically $L_H$-smooth if its Hessian is $L_H$-Lipschitz, i.e., for any $\bx$, $\by\in\cM$ with $d(\bx,\by)<\text{inj}(\cM)$,
\begin{equation}\label{lc0}
    \norm{\text{Hess} f(\by)-\Gamma_\bx^\by\text{Hess} f(\bx)\Gamma_\by^\bx}_{op}\leq L_H d(\bx,\by).
\end{equation}
\end{definition}
It is shown in the following lemma that if $f$ is $L_H$-smooth, then for any $\bx,\by\in\cM$ with $d(\bx,\by)<\text{inj}(\cM)$, we have
\begin{equation}\label{lc2}
    |f(\by)-(f(\bx)+\fprod{\eta,\text{grad} f(\bx)}_\bx+\frac{1}{2}\fprod{\eta,\text{Hess}f(\bx)[\eta]}_\bx)|\leq \frac{L_H}{6} d^3(\bx,\by)
\end{equation}
and
\begin{equation}\label{lc3}
              \|\text{grad} f(\by)-\Gamma_\bx^\by \text{grad} f(\bx)-\Gamma_\bx^\by \text{Hess}f(\bx)\eta\|_\by\leq \frac{L_H}{2} d^2(\bx,\by),
\end{equation}
where $\eta=\text{Exp}^{-1}_\bx (\by)$.

\begin{lemma}[\cite{agarwal2020adaptive}, Proposition 3.2]\label{lemma:hessian}
If $f$ is $H$-smooth with constant $L_H$, then \eqref{lc2} and \eqref{lc3} hold. 
\end{lemma}

The Lipschitz-type conditions above are parallel to the conditions in the Euclidean setting \cite{nester2006cubic}. In general, it is not trivial to verify these conditions, or even determine their parameters. However, we know there is a broad class of functions on Euclidean space, which satisfy the Lipschitz continuity-related conditions. 
We conjectured similar properties as the Euclidean setting would imply \eqref{lc0}, if $\cM$ is embedded in the Euclidean space. 
In \cite{absil2009optimization}, it was proven that if the manifold is compact and the function has Lipschitz continuous gradient, then \eqref{l1} holds. \cite{boumal2019global} proved that if the manifold is compact and the function has lipschitz continuous gradient and Hessian, then \eqref{lc1} and \eqref{lc2} hold. 
In the following lemma, it is shown that \eqref{lc0} holds under the same conditions.

\begin{lemma}\label{lemma:H_smoothness_proof}
If $\cM$ is a compact submanifold of the Euclidean space $\cE$ and $f(\bx)$ has Lipschitz continuous Hessian in $\cE$ in the Euclidean sense, then \eqref{lc0} is satisfied.
\end{lemma}

\begin{proof}
Denote the orthogonal projection operator onto $T_{\bx}\cM$, i.e. the tangent space of $\cM$ at $\bx$, by $P_{\bx}$. Denote the Euclidean gradient and Hessian by $\nabla f(\bx)$ and $\nabla^2 f(\bx)$ correspondingly. For any $\by$, such that $d(\bx,\by)<\text{inj}(\cM)$ and any $\xi\in T_{\by}\cM$, s.t. $\norm{\xi}=1$, we have
\begin{align*}
    \text{Hess} f(\by)[\xi]&=P_{\by}(D(\by\to P_{\by}\nabla f(\by))(\by)[\xi])\\
    &=P_{\by}(D(\by\to P_{\by})(\by)[\xi][\nabla f(\by)])+P_{\by} (\nabla^2 f(\by)[\xi])\\
    &\equiv A_1+B_1,
\end{align*}
The first equality follows from \eqref{Hessian_with_D} and the second equality comes from the chain rule and the fact that the projection operator is linear. Similarly, we have
\begin{align*}
    \Gamma_{\bx}^{\by}\text{Hess} f(\bx) [\Gamma_{\by}^{\bx}\xi] &= \Gamma_{\bx}^{\by} P_{\bx}(D(\bx\to P_{\bx}\nabla f(\bx))(\bx)[\Gamma_{\by}^{\bx}\xi])\\
    &=\Gamma^{\by}_{\bx} P_{\bx}(D(\bx\to P_{\bx})(\bx)[\Gamma_{\by}^{\bx}\xi][\nabla f(\bx)])+\Gamma^{\by}_{\bx}P_{\bx} (\nabla^2 f(\bx)[\Gamma_{\by}^{\bx}\xi])\\
    &\equiv A_2+B_2.
\end{align*}
First, to quantify $\norm{A_1-A_2}$, we have, 
\small
\begin{align}
	\norm{A_1-A_2}&=\norm{O_{A_1}[\nabla f(\by)+\nabla f(\bx)-\nabla f(\bx)]-O_{A_2}[\nabla f(\bx)]}\\
	&=\norm{O_{A_1}[\nabla f(\by)-\nabla f(\bx)]+(O_{A_1}-O_{A_2})[\nabla f(\bx)]}\\
	&\leq\norm{O_{A_1}[\nabla f(\by)-\nabla f(\bx)]}+\norm{(O_{A_1}-O_{A_2})[\nabla f(\bx)]}\\
	&\leq \norm{O_{A_1}}_{op}\cdot\norm{\nabla f(\by)-\nabla f(\bx)}+\norm{(O_{A_1}-O_{A_2})[\nabla f(\bx)]}\label{abcd}
\end{align}
\normalsize
where $O_{A_1}\triangleq P_{\by}(D(\by\to P_{\by})(\by)[\xi][\cdot])$ and $O_{A_2}\triangleq \Gamma^{\by}_{\bx} P_{\bx}(D(\bx\to P_{\bx})(\bx)[\Gamma_{\by}^{\bx}\xi][\cdot])$.

Due to the smoothness and compactness of $\cM$ and $\norm{\xi}=1$, $\norm{P_{\by}(D(\by\to P_{\by})(\by)[\xi][\cdot])}_{op}$ exists and is uniformly upper bounded, i.e. there exists a finite $M_1$ independent of $\bx$, $\by$ and $\xi$, s.t. $\norm{P_{\by}(D(\by\to P_{\by})(\by)[\xi][\cdot])}_{op}\leq M_1$ for any $\bx$, $\by\in\cM$ and $\xi$, s.t. $\norm{\xi}=1$.

For any $\bz$, such that $d(\bz,\by)<\text{inj}(\cM)$, define $Q_{\bz,\by,\xi}\triangleq\Gamma^{\by}_{\bz} P_{\bz}(D(\bz\to P_{\bz})(\bz)[\Gamma_{\by}^{\bz}\xi][\cdot])$. Note that $O_{A_1}=Q_{\by,\by,\xi}$ and $O_{A_2}=Q_{\bx,\by,\xi}$.
For fixed $\tilde{\bx}$, $\tilde{\by}$ and $\tilde{\xi}$, $Q_{\bz,\tilde{\by},\tilde{\xi}}[\nabla f(\tilde{\bx})]$ is a continuously differentiable function of $\bz$ based on the conditions that the manifold is smooth and $f(\bx)$ has Lipschitz continuous Hessian. Since $\bz$ belongs to a compact set, $Q_{\bz,\tilde{\by},\tilde{\xi}}[\nabla f(\tilde{\bx})]$ is Lipschitz continuous on $\bz$, i.e. 
\begin{equation}\label{Q_difference}
	\norm{Q_{\bx,\tilde{\by},\tilde{\xi}}[\nabla f(\tilde{\bx})]-Q_{\by,\tilde{\by},\tilde{\xi}}[\nabla f(\tilde{\bx})]}\leq M_{\tilde{\bx},\tilde{\by},\tilde{\xi}}\norm{\bx-\by}, \forall \bx,\by\in \cM
\end{equation}
where $M_{\tilde{\bx},\tilde{\by},\tilde{\xi}}$ is a finite constant depending on $\tilde{\bx},\tilde{\by},\tilde{\xi}$. Especially, due to the smoothness of manifold and the function $f(\bx)$ has Lipschitz continuous Hessian, we have a continuous mapping from $\tilde{\bx},\tilde{\by},\tilde{\xi}$ to $M_{\tilde{\bx},\tilde{\by},\tilde{\xi}}$. Since $\tilde{\bx},\tilde{\by}\in \cM$, which is a compact set and $\norm{\tilde{\xi}}=1$, we have a finite constant $M_2$, s.t. $M_{\tilde{\bx},\tilde{\by},\tilde{\xi}}\leq M_2$ for all ${\tilde{\bx},\tilde{\by},\tilde{\xi}}$. In \eqref{Q_difference}, letting $\bx=\tilde{\bx}$, $\by=\tilde{\by}$, we have, 
\begin{align}
	\norm{Q_{\tilde{\bx},\tilde{\by},\tilde{\xi}}[\nabla f(\tilde{\bx})]-Q_{\tilde{\by},\tilde{\by},\tilde{\xi}}[\nabla f(\tilde{\bx})]}\leq M_{\tilde{\bx},\tilde{\by},\tilde{\xi}}\norm{\tilde{\bx}-\tilde{\by}}\leq M_2\norm{\tilde{\bx}-\tilde{\by}}.
\end{align}
Due to the arbitrariness of $\tilde{\bx}$, $\tilde{\by}$ and $\tilde{\xi}$, we conclude the second term in \eqref{abcd}, $\norm{(O_{A_1}-O_{A_2})[\nabla f(\bx)]}\leq M_2\norm{\bx-\by}$.

On the other hand, the gradient of a twice continuously differentiable function on a compact manifold is Lipschitz continuous. 
 Therefore, there exists a finite $L_1$, s.t.
\begin{align}%\label{L_1}
    \norm{A_1-A_2} \leq M_1\norm{\nabla f(\by)-\nabla f(\bx)}+M_2\norm{\bx-\by}
    \leq (M_1\cdot L_1+M_2) \norm{\by-\bx}
    \leq (M_1\cdot L_1+M_2)d(\bx,\by)\label{A_1-A_2},
\end{align}
where $d(\bx,\by)$ is the Riemannian distance between $\bx$ and $\by$. The third inequality holds since the manifold is embedded in the Euclidean space.

Second, to quantify $\norm{B_1-B_2}$, we define 
\begin{equation}
	R_{\by,\xi}(\bz)\triangleq \Gamma^{\by}_{\bz}P_{\bz} (\nabla^2 f(\bz)[\Gamma_{\by}^{\bz}\xi]).
\end{equation}

Fixing $\by$,$\xi$ to be $\tilde{\by}$ and $\tilde{\xi}$, $R_{\tilde{\by},\tilde{\xi}}(\bz)$ is Lipschitz continuous on $\bz$ due to the smoothness of the manifold and $\nabla^2 f(\bz)$ is Lipschitz continuous. Therefore, there exists a constant $N_{\tilde{\by},\tilde{\xi}}$ depending on $\tilde{\by}$ and $\tilde{\xi}$, s.t. $\norm{R_{\tilde{\by},\tilde{\xi}}(\bx)-R_{\tilde{\by},\tilde{\xi}}(\by)}\leq N_{\tilde{\by},\tilde{\xi}}\norm{\bx-\by}$ for all $\bx$, $\by\in\cM$. Especially, there is a continuous mapping from $\tilde{\by}$, $\tilde{\xi}$ to $N_{\tilde{\by},\tilde{\xi}}$. Since $\tilde{\by}$, $\tilde{\xi}$ are from compact sets, there exists a finite constant $M_3$, s.t. $\norm{R_{\tilde{\by},\tilde{\xi}}(\bx)-R_{\tilde{\by},\tilde{\xi}}(\by)}\leq N_{\tilde{\by},\tilde{\xi}}\norm{\bx-\by}\leq M_3 \norm{\bx-\by}$ for all $\bx$, $\by\in\cM$.

Letting $\bx=\tilde{\bx}$, $\by=\tilde{\by}$, and due to the arbitrariness of $\tilde{\bx}$, $\tilde{\by}$ and $\tilde{\xi}$, we have 
\begin{equation}\label{B_1-B_2}
	\norm{B_1-B_2}=\norm{R_{\by,\xi}(\by)-R_{\by,\xi}(\bx)}\leq M_3 \norm{\bx-\by}\leq M_3\cdot d(\bx,\by).
\end{equation}
Combining \eqref{A_1-A_2}, \eqref{B_1-B_2}, there exists a finite $L\triangleq M_1\cdot L_1+M_2+M_3$, s.t.
\begin{equation*}
    \norm{\text{Hess} f(\by)[\xi]-\Gamma_{\bx}^{\by}\text{Hess} f(\bx) [\Gamma_{\by}^{\bx}\xi]}\leq \norm{A_1-A_2}+\norm{B_1-B_2}\leq L\cdot d(\bx, \by)
\end{equation*}
Since $\xi$ is an arbitrary tangent vector, we have, 
\begin{equation*}
    \norm{\text{Hess} f(\by)-\Gamma_{\bx}^{\by}\text{Hess} f(\bx) \Gamma_{\by}^{\bx}}_{op}\leq L\cdot d(\bx, \by).
\end{equation*}
\end{proof}

\section{Complexity Analysis of the Proposed Algorithm}\label{sec:con_analys}

\begin{definition}[Optimal gap] For function $F(\cdot)$ and the initial point $\bx_0\in\cM$, define
\begin{equation}\label{delta}
    \Delta_F \triangleq F(\bx_0)-F^*,
\end{equation}
where $F^*=\inf_{\bx\in\cM} F(\bx)$. 
\end{definition}
Without loss of generality, we assume $\Delta_F < +\infty$ throughout this paper.
\begin{definition}[$(\epsilon,\delta)$-second-order stationary point] \label{def:eps_delta}
$\bx$ is a second-order stationary point of the function $F:\cM\rightarrow\mR$ if $\norm{\text{grad} F(\bx)}\leq \epsilon$ and $\lambda_{\min} (\text{Hess} F(\bx))\geq -\delta$ where $\text{grad} F(\bx)$ and $\text{Hess} F(\bx)$ are the Riemannian gradient and Hessian of $F$ at $\bx$ and $\lambda_{\min} (\text{Hess} F(\bx))\triangleq\inf_{\eta\in T_{\bx}\cM} \{\fprod{\text{Hess} F(\bx)\eta,\eta}_\bx/\norm{\eta}^2\}$.
\end{definition}
As in \cite{nester2006cubic}, we define
\begin{equation}\label{mux}
    \mu (\bx) \triangleq\max \{\norm{\text{grad} F(\bx)}^{3/2}, -\frac{\lambda^3_{min}(\text{Hess}F(\bx))}{L_H^{3/2}} \}.
\end{equation}
In particular, according to the definition \eqref{mux}, $\mu(\bx)\leq \epsilon^{3/2}$ holds if and only if 
\begin{equation}\label{mux_expand}
\norm{\text{grad} F(\bx)}\leq \epsilon, \quad \lambda_{\min} (\text{Hess} F(\bx))\geq -\sqrt{L_H \epsilon}.
\end{equation}
Therefore, in order to find an $(\epsilon, \sqrt{L_H \epsilon})-$ approximate local minimum of the function defined over $\cM$, it suffices to find $\bx\in\cM$ such that $\mu(\bx)\leq \epsilon^{3/2}$.

%%%%%% Assumption Part
\begin{assumption}\label{assumption:smooth}
We assume that the objective function $F$ is bounded below, and its components $f_i,\ i=1,\cdots,N$ are twice continuously differentiable and they are g- and H-smooth.
\end{assumption}
\begin{assumption}\label{assumption:bounded}
	We assume that either i) functions $f_i,\ i=1,\cdots,N$ are Lipschitz continuous, or ii)  functions $f_i,\ i=1,\cdots,N$ are continuously differentiable and the manifold $\cM$ is compact.
	%that the norm of Riemannian gradient of $f_i$ is bounded.
\end{assumption}
\begin{remark}
	The g-smoothness of $f_i,\ i=1,\cdots,N$, in Assumption~\ref{assumption:smooth} implies that $\norm{\text{Hess} F(\bx)}_{op}$ is bounded. Furthermore, Assumption~\ref{assumption:bounded} implies that $\norm{\text{grad} F(\bx)}$ is bounded either by Lipschitz continuity of $f_i,\ i=1,\cdots,N,$ or by the Weierstrass theorem~\cite{rudin1964principles}. Hence, under Assumptions~\ref{assumption:smooth} and~\ref{assumption:bounded}, and based on the fact that the parallel transport is isometric, there exist two positive constants $c_g$ and $c_H$, such that 
\begin{equation}\label{assumptionremark:bound_hessian_grad}
	\norm{\bv^s_t}\leq c_g \text{ and } \norm{U^s_t}_{op}\leq c_H. 
\end{equation}
\end{remark}

While the above two assumptions are mainly related  to the objective function, the following three assumptions are related to the manifold.
\begin{assumption}\label{assumption:embedded}
We assume that $\cM$ is embedded in a vector space, e.g., Euclidean space. For the ease of presentation, we assume $\cM \subseteq \mR^{m\times n}$.
\end{assumption}
\begin{remark}
	Under Assumption~\ref{assumption:embedded}, the Riemannian metric $g_{\bx}(\cdot,\cdot)$ on the tangent space $T_{\bx}\cM$ is the restriction of the Euclidean metric. The norm induced by the Riemannian metric $\norm{\cdot}_{\bx}$ is the Euclidean norm.  
\end{remark}
\begin{assumption}\label{assumption:positiveinj}
We assume that the manifold has positive injectivity radius, i.e. $\text{inj}(\cM)\in(0,\infty]$-see Definition~\ref{def:injec_radius}. 
\end{assumption}
\begin{remark}
To provide few examples, the unit sphere has injectivity radius equal to $\pi$, Hadamard manifolds and Euclidean spaces have infinite injectivity radius, and compact Riemannian manifolds have positive injectivity radius~\cite{chavel2006riemannian}. Note that the Assumption~\ref{assumption:positiveinj} implies that the manifold is complete. 
\end{remark}
\begin{assumption}\label{assumption:curvature}
	We assume the sectional curvature of the Riemannian manifold $\cM$ is lower-bounded by $\upkappa$ - see \cite{lee2018introduction} for the definition of the sectional curvature. 
\end{assumption}
\begin{remark}
	Some manifolds that satisfy Assumption~\ref{assumption:curvature} include rotation group, hyperbolic manifold, the sphere, orthogonal groups, real projective space, Grassmann manifold, Stiefel manifold and compact subsets of the cone of positive definite matrices (see \cite{bonnabel2013stochastic,sra2015conic,boumal2020introduction}).
\end{remark}
%\todo{What manifolds satisfy Assumptions 3-5 simultaneously?}

Following the literature of the Newton method with cubic regularization \cite{nester2006cubic,cartis2011adaptive}, we define 
\begin{equation}\label{original_cubic}
	\tilde{m}(\bh)=\fprod{\text{grad} f,\bh}+\frac{1}{2}\fprod{\text{Hess} f[\bh],\bh}+\frac{\sigma}{6}\|\bh\|^3, \bh\in T_{\bx}\cM,
\end{equation}
which can be regarded as a cubic regularization of locally quadratic approximation of function $f$ -- see \cite{agarwal2018adaptive}. From \eqref{lc2}, we have $f(\text{Exp}_{\bx}(\eta))\leq \tilde{m}(\eta)$, for $\forall\eta\in T_{\bx} \cM$, if $\sigma\geq L_H$. From \eqref{original_cubic}, we define
\begin{equation}\label{m_cubic}
\bh^s_t=\argmin_{\bh\in T_{\bx}\cM}m^s_t(\bh),
\end{equation}
where
\begin{equation} 
m^s_t(\bh)\triangleq\fprod{\bv^s_t,\bh}+\frac{1}{2}\fprod{\bU^s_t[\bh],\bh}+\frac{\sigma}{6}\|\bh\|^3
\end{equation}
and $\bv^s_t$ and $\bU^s_t$ are the approximated Riemannian gradient and Hessian operator of the objective function. Generally, \eqref{m_cubic} and \eqref{original_cubic} are not convex problems. \cite{nester2006cubic} proposed a way to transform these subproblems into convex programs in one variable. Recently, results in \cite{carmon2019gradient} show that under mild conditions gradient descent approximately finds the \emph{global minimum} with the rate of $O(\epsilon^{-1}\log(1/\epsilon))$. \cite{cartis2011adaptive2} propose a Lanczos-based method to minimize \eqref{original_cubic} exactly. The gradient, conjugate gradient and Newton methods to minimize \eqref{original_cubic} are available in the software package provided in \cite{manopt}. 

\vspace{0.1cm}
%The framework of our analysis follows that of~\cite{pmlr-v80-zhou18d}. 
We first provide some preliminary lemmas. Lemma~\ref{lemma:optimal} provides three identities that are used in the proofs of following lemmas. These identities are typical in the cubic regularization literature~\cite{nester2006cubic,cartis2011adaptive,cartis2011adaptive2}. Lemma~\ref{lemma:bound_stepsize_via_sigma} provides an upper bound on $\norm{\bh_s^t}$ which then provides a (lower) bound on the cubic regularization parameter $\sigma$ to have the iterates close enough to the epoch points. Finally, Lemmas~\ref{expectation1} and \ref{lemma:innerproduct} provide upper bounds on the norm difference of $\text{grad} F$ and $\text{Hess} F$ with their variance-reduced estimators, and on the inner products $\fprod{\text{grad} F(\bx^s_t) - \bv^s_t, \eta}$ and $\fprod{(\text{Hess} F(\bx^s_t) - U^s_t)[\eta], \eta}$ for $\eta\in T_{\bx^s_t}\cM$ which are used in the proofs of the main theorems.

\begin{lemma}\label{lemma:optimal}
Under Assumptions~\ref{assumption:embedded}, for the semi-stochastic gradient and Hessian, we have
% \begin{equation}\label{optimal:1}
%     \bv^s_t+\bU^s_t\circ P_{\bx^s_t}\bh^s_t+\frac{M}{2}\norm{\bh^s_t}\bh^s_t=0,
% \end{equation}
% \begin{equation}\label{optimal:2}
%     \bU^s_t\circ P_{\bx^s_t}+\frac{M}{2}\norm{\bh^t_s}\bI\succeq 0,
% \end{equation}
% \begin{equation}\label{optimal:3}
%     \fprod{\bv^s_t,\bh^s_t}+\frac{1}{2}\fprod{\bU^s_t\circ P_{\bx^s_t} \bh^s_t,\bh^s_t}+\frac{M}{6}\norm{\bh^s_t}^3\leq-\frac{M}{12}\norm{\bh^s_t}^3.
% \end{equation}
\begin{equation}\label{optimal:1}
    \bv^s_t+\bU^s_t\bh^s_t+\frac{\sigma}{2}\norm{\bh^s_t}\bh^s_t=0,
\end{equation}
\begin{equation}\label{optimal:2}
    \bU^s_t+\frac{\sigma}{2}\norm{\bh^t_s}\bI\succeq 0,
\end{equation}
\begin{equation}\label{optimal:3}
    \fprod{\bv^s_t,\bh^s_t}+\frac{1}{2}\fprod{\bU^s_t \bh^s_t,\bh^s_t}+\frac{\sigma}{6}\norm{\bh^s_t}^3\leq-\frac{\sigma}{12}\norm{\bh^s_t}^3.
\end{equation}
\end{lemma}
\begin{proof}
	(sketch) Under Assumption~\ref{assumption:embedded}, i.e. the manifold is embedded in the Euclidean space, then the tangent space $T_{\bx}\cM$ in~\eqref{m_cubic} is isomorphic to subspace of the Euclidean space. Hence, the proof follows, e.g., from that of Lemma 24 in~\cite{pmlr-v80-zhou18d}. Indeed, the proof of \eqref{optimal:1} directly follows from the first-order optimality condition for a stationary point of \eqref{m_cubic}. The inequality \eqref{optimal:2} relies on the fact that $\bh^s_t$ is a global minimizer which will not hold when solving \eqref{m_cubic} inexactly. The proof of \eqref{optimal:3} is based on \eqref{optimal:2} and \eqref{optimal:1}.
\end{proof}

\begin{lemma}\label{lemma:bound_stepsize_via_sigma}
	Under Assumptions~\ref{assumption:smooth}-\ref{assumption:positiveinj}, given a constant $C>0$ and $\sigma>\frac{2(c_g+C\cdot c_H)}{C^2}$, we have $\norm{\bh^s_t}< C$. 
\end{lemma}
\begin{proof}
	Multiplying both sides of \eqref{optimal:1} by $\bh^s_t$, we obtain $\fprod{\bv^s_t,\bh^s_t}+\fprod{\bU^s_t \bh^s_t,\bh^s_t}+\frac{\sigma}{2}\norm{\bh^s_t}^3=0$.
	By Cauchy--Schwarz inequality, we have 
	$
		\frac{\sigma}{2}\norm{\bh^s_t}^3\leq\norm{\bv^s_t}\cdot\norm{\bh^s_t}+\norm{\bU^s_t}_{op}\cdot\norm{\bh^s_t}^2.
	$
	Dividing both sides by $\norm{\bh^s_t}$ and based on \eqref{assumptionremark:bound_hessian_grad}, we have 
	$
		\frac{\sigma}{2}\norm{\bh^s_t}^2-c_H\cdot \norm{\bh^s_t}-c_g\leq 0,
	$
	which implies 
	\begin{equation}\label{lemmaproof:1}
		\norm{\bh^s_t}\leq\frac{c_H+\sqrt{c_H^2+2\sigma c_g}}{\sigma}.
	\end{equation}
	Note that the right hand side of~\eqref{lemmaproof:1} is a monotonic decreasing function on $\sigma$. Hence, if $\sigma>\frac{2(c_g+C\cdot c_H)}{C^2}$, the right hand side of~\eqref{lemmaproof:1} is upper bounded by $C$, which implies $\norm{\bh^s_t} <C$. 
\end{proof}

\begin{remark}\label{remark:control_stepsize}
	In Lemma~\ref{lemma:bound_stepsize_via_sigma} as well as Lemma~\ref{lemma:bound_stepsize_via_sigma2}, we set $C=\frac{\text{inj}(\cM)}{T}$, where $T$ is the epoch length in Algorithm~\ref{alg:svr_cubic}. Then, for any epoch $s$ and iteration $t\in\{0,\cdots,T-1\}$, we have
	\begin{equation}\label{eq:UB_dist_iterate_epoch}
		d(\hat{\bx}^s,\bx^s_t)\leq \sum_{i=1}^{t}d(\bx^s_{i-1},\bx^s_i)\leq \sum_{i=1}^{t}\norm{\bh^s_i}<\text{inj}(\cM). 
	\end{equation}
	This inequality guarantees line 7 in Algorithm~\ref{alg:svr_cubic} is attained. In the following, we assume $\sigma$ is large enough such that
	\begin{equation}\label{cond:sigma1}
		\sigma> \frac{2(c_g T^2+\text{inj}(\cM)c_H T)}{(\text{inj}(\cM))^2},
	\end{equation}
	hence, the distance between the iterate $\bx^s_t$ and $\hat{\bx}^s$ is smaller than $\text{inj}(\cM)$. 
\end{remark}

\noindent In the proof of Lemmas \ref{expectation1} and \ref{lemma:innerproduct}, the crucial identity is the Lyapunov Inequality in \cite{durrett2019probability} and a couple of matrix concentration inequalities~\cite{mackey2014matrix}. Since there is no essential difference between Lemmas~25-27 in \cite{pmlr-v80-zhou18d} and our setting, we refer readers to \cite{pmlr-v80-zhou18d} and references therein.

\begin{lemma}\label{expectation1}
Under Assumptions~\ref{assumption:smooth}-\ref{assumption:positiveinj}, for the semi-stochastic gradient $\bv^s_t$ and semi-stochastic Hessian $\bU^s_t$, we have 
\begin{align}
    \mE_{I_g} \norm{\text{grad} F(\bx^s_t)-\bv^s_t}^{3/2} &\leq \frac{L_H^{3/2}}{b_g^{3/4}}\norm{\text{Exp}^{-1}_{\hat{\bx}^s}(\bx^s_t)}^3,\\
    \mE_{I_h} \norm{\text{Hess} F(\bx^s_t)-\bU^s_t}_{op}^3 &\leq 64L_H^3(\rho+\rho^2)^3 \norm{\text{Exp}^{-1}_{\hat{\bx}^s}(\bx^s_t)}^3,
\end{align}
where $\rho=\sqrt{\frac{2e\log mn}{b_h}}$. 
\end{lemma}

\begin{lemma}\label{lemma:innerproduct}
% For the semi-stochastic gradient and Hessian defined in and any $\eta\in T_{\bx^s_t} \cM$, we have 
For any $\eta\in T_{\bx^s_t}\cM$ and $M>0$, we have
\begin{equation}
    \fprod{\text{grad} F(\bx^s_t)-\bv_t^s,\eta}\leq\frac{M}{27}\norm{\eta}^3+\frac{2\norm{\text{grad} F(\bx^s_t)-\bv_t^s}^{3/2}}{M^{1/2}},
\end{equation}
\begin{equation}
    \fprod{(\text{Hess} F(\bx^s_t)-\bU_t^s)[\eta],\eta}\leq\frac{2M}{27}\norm{\eta}^3+\frac{27}{M^2}\norm{\text{Hess} F(\bx^s_t)-\bU_t^s}_{op}^3.
\end{equation}
\end{lemma}

The following Lemmas~\ref{lemma:small1} and \ref{lemma:small2} provide an upper bound on $\norm{\text{grad}F}$ and a lower bound on $\lambda_{\min}(\text{Hess}F)$, respectively.

\begin{lemma}\label{lemma:small1}
  Under Assumptions~\ref{assumption:smooth}-\ref{assumption:positiveinj}, if $\sigma\geq 2 L_H$ and also satisfies~\eqref{cond:sigma1}, then for any $\bh\in T_{\bx^s_t}\cM$ such that $\norm{\bh}< \text{inj}(\cM)$, we have 
  \begin{equation}\label{lemm44}
       \norm{\mbox{grad F} (\text{Exp}_{\bx^s_t}(\bh))}\leq \sigma\norm{\bh}^2+\norm{\text{grad} F(\bx^s_t)-\bv^s_t}+\frac{1}{\sigma} \norm{\text{Hess} F(\bx^s_t)-\bU^s_t}^2_{op}+\norm{\nabla m^s_t(\bh)}.
  \end{equation}
\end{lemma}
\begin{proof}
~For simplicity, we denote $\text{Exp}_{\bx^s_t}(\bh)$ by $\by$, the parallel transport operator $\Gamma_{\bx^s_t}^{\by}$ by $\Gamma$, and $\Gamma_{\by}^{\bx^s_t}$ by $\Gamma^{-1}$. We have
\small
\begin{align*}
    \norm{\text{grad} F(\by)}&=\norm{\Gamma^{-1} \text{grad} F(\by)}\\
    &=\norm{\Gamma^{-1}\text{grad} F(\by)-\text{grad}F(\bx^s_t)-\text{Hess} F(\bx^s_t)\bh+\bv^s_t+\bU^s_t \bh+\frac{\sigma\norm{\bh}}{2}\bh\\ &+(\text{grad} F(\bx^s_t)-\bv^s_t)+(\text{Hess} F(\bx^s_t)-\bU^s_t)\bh-\frac{\sigma\norm{\bh}}{2}\bh}\\
    &\leq \norm{\Gamma^{-1}\text{grad} F(\by)-\text{grad}F(\bx^s_t)-\text{Hess} F(\bx^s_t)\bh}+\norm{\bv^s_t+\bU^s_t \bh+\frac{\sigma\norm{\bh}}{2}\bh}\\
    &+\norm{\text{grad} F(\bx^s_t)-\bv^s_t}+\norm{(\text{Hess} F(\bx^s_t)-\bU^s_t)\bh}+\frac{\sigma \norm{\bh}^2}{2}.
\end{align*}
\normalsize
Due to the isometric property of $\Gamma$ and Lemma~\ref{lemma:hessian}, we have 
\small
\begin{align*}
    \norm{\Gamma^{-1}\text{grad} F(\by)-\text{grad}F(\bx^s_t)-\text{Hess} F(\bx^s_t)\bh}&=\norm{\text{grad} F(\by)-\Gamma\text{grad}F(\bx^s_t)-\Gamma\text{Hess} F(\bx^s_t)\bh}\\
    &\leq \frac{L_H}{2} \norm{\bh}^2 \leq \frac{\sigma}{4} \norm{\bh}^2,
\end{align*}
\normalsize
where the last inequality follows from the condition $\sigma \geq 2 L_H$. From the definition of $m^s_t(\cdot)$ in~\eqref{m_cubic}, we have  
$
    \norm{\bv^s_t+\bU^s_t \bh+\frac{\sigma\norm{\bh}}{2}\bh}=\norm{\nabla m^s_t(\bh)}.
$
Note that 
\begin{align*}
    \norm{(\text{Hess} F(\bx^s_t)-\bU^s_t)\bh}&\leq \norm{\text{Hess} F(\bx^s_t)-\bU^s_t}_{op}\norm{\bh}
    \leq \frac{1}{\sigma}\norm{\text{Hess} F(\bx^s_t)-\bU^s_t}_{op}^2+\frac{\sigma}{4}\norm{\bh}^2,
\end{align*}
where the last inequality is due to Young's inequality. Combining these results, the proof of \eqref{lemm44} is completed.
\end{proof}

\begin{lemma}\label{lemma:small2}
  Under Assumptions~\ref{assumption:smooth}-\ref{assumption:positiveinj}, if $\sigma\geq 2 L_H$ and also satisfies~\eqref{cond:sigma1}, then for any $\bh\in T_{\bx^s_t}\cM$ such that $\norm{\bh}<\text{inj}(\cM)$, we have  
  \begin{equation}
      -\lambda_{\min} (\text{Hess} F(\text{Exp}_{\bx^s_t}(\bh)) \leq \sigma \norm{\bh} +\norm{\text{Hess} F(\bx^s_t)-\bU^s_t}_{op} + \frac{\sigma}{2} |\norm{\bh}-\norm{\bh^s_t}|,
  \end{equation}
  where $\lambda_{\min} (\text{Hess} F(\bx))$ is defined as $\lambda_{\min} (\text{Hess} F(\bx))=\inf_{\eta\in T_{\bx}\cM} \{\frac{\fprod{\text{Hess} F(\bx)\eta,\eta}}{\norm{\eta}}\}$.
\end{lemma}
\begin{proof}
Denote $\text{Exp}_{\bx^s_t}(\bh)$ by $\by$ and $\bx^s_t$ by $\bx$.  Furthermore, let $\mathbf{I}_{\bx}$ denotes the identity operator at $\bx$, i.e., $\mathbf{I}_{\bx}(\eta)=\eta$ for any $\eta\in T_{\bx} \cM$. We have 
\begin{align*}
    H_{\by} &\succeq \Gamma_{\bx}^{\by} H_{\bx} \Gamma_{\by}^{\bx}-L_H \norm{\bh} \mathbf{I}_{\by}\\
     & \succeq \Gamma_{\bx}^{\by} U^s_t \Gamma_{\by}^{\bx}-\norm{\Gamma_{\bx}^{\by} H_{\bx} \Gamma_{\by}^{\bx}-\Gamma_{\bx}^{\by} U^s_t \Gamma_{\by}^{\bx}}_{op} \mathbf{I}_{\by}-L_H \norm{\bh} \mathbf{I}_{\by}\\
    & \succeq -\frac{\sigma}{2}\norm{\bh^s_t} \mathbf{I}_{\by}-\norm{H_{\bx}-U^s_t}_{op} \mathbf{I}_{\by}-L_H\norm{\bh} \mathbf{I}_{\by},
\end{align*}
where the first inequality follows from the H-smooth assumption \eqref{lc0}, the second inequality follows from the definition of the operator norm and triangle inequality, and the third inequality follows from the isometric property of the parallel transport $\Gamma$ and the following argument.
Assume that $\Gamma_{\bx}^{\by} U^s_t \Gamma_{\by}^{\bx}\succeq -\frac{\sigma}{2}\norm{\bh^s_t} \mathbf{I}_{\by}$ does not hold, then there exists $\xi\in T_{\by}\cM$, s.t. $\fprod{\xi, \Gamma_{\bx}^{\by} U^s_t \Gamma_{\by}^{\bx}\xi}+\frac{\sigma}{2}\norm{\bh^s_t}\cdot \norm{\xi}^2<0$. Denote the $\Gamma_{\by}^{\bx}\xi$ by $\eta$, we have
$\fprod{\eta,\bU^s_t \eta}+\frac{\sigma}{2}\norm{\bh^s_t}\cdot \norm{\eta}^2=\fprod{\xi, \Gamma_{\bx}^{\by} U^s_t \Gamma_{\by}^{\bx}\xi}+\frac{\sigma}{2}\norm{\bh^s_t}\cdot \norm{\xi}^2<0$,
which contradicts \eqref{optimal:2}.
Therefore, we have 
\begin{align*}
    -\lambda_{\min} (H_{\by}) &\leq \frac{\sigma}{2} \norm{\bh^s_t}+\norm{H_{\bx}-U^s_t}_{op}+ L_H\norm{\bh}\\
    &=\frac{\sigma}{2} (\norm{\bh^s_t}-\norm{\bh})+\norm{H_{\bx}-U^s_t}_{op}+ (L_H+\sigma/2)\norm{\bh}\\
    &\leq \sigma\norm{\bh}+\norm{H_{\bx}-U^s_t}_{op}+ \frac{\sigma}{2} |\norm{\bh^s_t}-\norm{\bh}|,
\end{align*}
where the last inequality holds because $L_H \leq \sigma/2$. 
\end{proof}

Combining Lemmas \ref{lemma:small1} and \ref{lemma:small2} and the definition of $\mu(\bx)$ in \eqref{mux}, we have the following result.
\begin{lemma}\label{lemma:mu_after_two_small}
Under Assumptions~\ref{assumption:smooth}-\ref{assumption:positiveinj}, setting $\sigma=\bar{k} L_H$ such that $\bar{k}\geq 2$ and $\sigma$ satisfying~\eqref{cond:sigma1}, then for any $\bh\in T_{\bx^s_t}\cM$ such that $\norm{\bh}<\text{inj}(\cM)$, we have 
  \begin{align*}
      \mu(\text{Exp}_{\bx^s_t}(\bh))&\leq 9\bar{k}^{3/2}[\sigma^{3/2}\norm{\bh}^3+\norm{\text{grad} F(\bx^s_t)-\bv^s_t}^{3/2}+\sigma^{-3/2}\norm{\text{Hess} F(\bx^s_t)-\bU^s_t}^{3}\\
      &+\norm{\nabla m^s_t(\bh)}^{3/2}+\frac{\sigma^{3/2}}{8}|\norm{\bh}-\norm{\bh^s_t}|^3].
  \end{align*}
 \end{lemma}
\begin{proof}
The proof follows from that of Lemma~\ref{lemma2:mu_after_two_small}.
\end{proof}

Next, We present the following result from \cite{zhang2016riemannian}. This inequality extends the law of cosines from Euclidean space to Riemannian space, which is fundamental to carry out non-asymptotic analysis for Riemannian optimization. The resulting inequality is used in the proof of Lemma~\ref{lemma:tri}.
\begin{lemma}[\cite{zhang2016riemannian}, Lemma 5]\label{lemma:tria}
  If $a$, $b$ and $c$ are the side lengths of a geodesic triangle in an Alexandrov space with curvature lower-bounded by $\upkappa$, and $A$ is the angle between sides $b$ and $c$, then 
  \begin{equation}
      a^2 \leq \frac{\sqrt{|\upkappa|}c}{\tanh{\sqrt{|\upkappa|}c}}b^2+c^2-2bc\cos{A}.
  \end{equation}
\end{lemma}
Lemmas~\ref{lemma:tri} and \ref{lemma:constant_c_t} below are used in the proof of the first main result presented in Theorem~\ref{main1}.
\begin{lemma}\label{lemma:tri}
Let $\zeta\triangleq\sqrt{|\upkappa|}\text{inj}(\cM)/{\tanh{\sqrt{|\upkappa|}\text{inj}(\cM)}}$ if $\upkappa< 0$ and $\zeta\triangleq 1$, o.w..
%See Assumption~\ref{assumption:curvature} for definition of $\upkappa$. 
Then, under Assumptions~\ref{assumption:smooth}-\ref{assumption:curvature}, for any $\bh\in T_{\bx^s_t}\cM$ such that $\norm{\bh}<\text{inj}(\cM)$ and $T\geq 2$, we have 
  \begin{equation}
      \norm{\text{Exp}^{-1}_{\hat{\bx}^s}(\text{Exp}_{\bx^s_t}(\bh))}^3 \leq 2(\sqrt{\zeta-1}+1)^3 T^2 \norm{\bh}^3+(1+\frac{3}{T})\norm{\text{Exp}^{-1}_{\hat{\bx}^s}(\bx^s_t)}^3.
  \end{equation}
\end{lemma}
\begin{proof}
Let $g(t)=t/\tanh{t}$ which is non-decreasing on $[0,\sqrt{|\upkappa|}D]$ and $g(t)\geq 1$. For simplicity, denote $\norm{\text{Exp}^{-1}_{\hat{\bx}^s}(\text{Exp}_{\bx^s_t}(\bh))}$, $\norm{\bh}$ and $\norm{\text{Exp}^{-1}_{\hat{\bx}^s}(\bx^s_t)}$ by $a$, $b$ and $c$, respectively. 
By Lemma~\ref{lemma:tria}, we have 
\begin{align*}
    %\norm{\text{Exp}^{-1}_{\hat{\bx}^s}(\text{Exp}_{\bx^s_t}(\bh))}^2 &\leq \frac{\sqrt{|\upkappa|}c}{\tanh{\sqrt{|\upkappa|}c}}
    a^2 &\leq \frac{\sqrt{|\upkappa|}c}{\tanh{\sqrt{|\upkappa|}c}}b^2+c^2-2bc\cos{A}
    \leq (b+c)^2+(\zeta-1)b^2
    \leq [(\sqrt{\zeta-1}+1)b+c]^2.
\end{align*}
Therefore, 
\footnotesize
\begin{align*}
    a^3 &\leq [(\sqrt{\zeta-1}+1)b+c]^3\\
    &= (\sqrt{\zeta-1}+1)^3 b^3 +3T^{1/3}(\sqrt{\zeta-1}+1)^2 b^2 \frac{c}{T^{1/3}}+3 T^{2/3}(\sqrt{\zeta-1}+1)b \frac{c^2}{T^{2/3}}+c^3\\
    &\leq (\sqrt{\zeta-1}+1)^3 b^3+3(\frac{2}{3}[T^{1/3}(\sqrt{\zeta-1}+1)^2b^2]^{3/2}+\frac{1}{3}\frac{c^3}{T})
    +3(\frac{1}{3}[T^{2/3}(\sqrt{\zeta-1}+1)b]^3+\frac{2c^{3}}{3T})+c^3\\
    &=(\sqrt{\zeta-1}+1)^3(1+2\sqrt{T}+T^2)b^3+(1+\frac{3}{T})c^3\\
    &\leq 2(\sqrt{\zeta-1}+1)^3 T^2 b^3 +(1+\frac{3}{T})c^3,
\end{align*}
where the second inequality follows from Young's inequality  and the last inequality follows from the fact that $1+2\sqrt{T}+T^2\leq 2T^2$ when $T\geq 2$.
\end{proof}

\begin{lemma}\label{lemma:constant_c_t}
 Define the series $c_t\triangleq c_{t+1}(1+3/T)+\sigma[500T^3(\sqrt{\xi-1}+1)^3]^{-1}$ for $0\leq t\leq T-1$ and $c_T=0$. Then for any $1\leq t\leq T$, we have
  \begin{equation}
      \sigma/24-2c_t (\sqrt{\xi-1}+1)^3T^2 \geq 0. 
  \end{equation}
\end{lemma}
\begin{proof}
~Assuming $c_t +q =p(c_{t+1}+q)$, we can derive $p=1+3/T$ and $q=\sigma[1500T^2(\sqrt{\xi-1}+1)^3]^{-1}$. Furthermore, given $c_T=0$ by induction, we have $c_t=(p^{T-t}-1)q$. Therefore,
	\begin{equation}
		2c_t (\sqrt{\xi-1}+1)^3T^2=((1+\frac{3}{T})^{T-t}-1)\frac{\sigma}{750}\leq (1+\frac{3}{T})^T \frac{\sigma}{750}\leq\frac{\sigma}{24},
	\end{equation}
	where the last inequality follows from the fact $(1+3/T)^T \leq 27$.
\end{proof}

Theorem~\ref{main1} below presents our first main result. It provides the convergence rate of the R-SVRC algorithm when the cubic regularized Newton subproblem is solved \emph{exactly}.
%In Theorem~\ref{main2}, we present the other one \blue{analogous} to the scenario where the cub-regularized subproblem~\eqref{m_cubic} is solved inexactly.
\begin{theorem}\label{main1}
Under Assumptions~\ref{assumption:smooth}-\ref{assumption:curvature}, suppose that the cubic regularization parameter $\sigma$ in Algorithm~\ref{alg:svr_cubic} is fixed and satisfies $\sigma=\bar{k} L_H$, where $L_H$ is the Hessian Lipschitz parameter according to \eqref{lc0}, $\bar{k}\geq 2$ and $\sigma$ satisfies~\eqref{cond:sigma1}. Furthermore, assume that the batch size parameters $b_g$ and $b_h$ satisfy
\begin{equation}\label{condition:batch_sizes}
    b_g \geq \frac{3000^{4/3}T^4 (\sqrt{\xi-1}+1)^4}{\bar{k}^2}, \quad b_h \geq \frac{e\log d}{(\sqrt{\frac{\bar{k}}{193T(\sqrt{\zeta-1}+1)}+\frac{1}{8}}-\frac{1}{2\sqrt{2}})^2}, 
\end{equation}
where $T\geq 2$ is the length of the inner loop, $e$ is the Euler's number and $d=mn$ is the dimension of the problem. Then, we have 
\begin{equation}
    \mE [\mu(\bx_{out})]\leq \frac{240 \bar{k}^2 L_H^{1/2}\Delta_F}{S T},
\end{equation}
where $\mu(\bx)$ is defined in \eqref{mux}.
\end{theorem}

\begin{proof}
First, we upper bound $F(\bx_{t+1}^s)$ as follows:
\begin{align*}
    F(\bx_{t+1}^s)&\leq F(\bx_t^s)+\fprod{\text{grad} F(\bx_t^s),\bh^s_t}+\frac{1}{2}\fprod{H_{\bx_t^s}[\bh^s_t],\bh^s_t}+\frac{L_H}{6}\norm{\bh^s_t}^3\\
    &=F(\bx_t^s)+\fprod{\text{grad} F(\bx_t^s)-\bv_t^s,\bh^s_t}+\frac{1}{2}\fprod{(H_{\bx_t^s}-\bU_t^s)[\bh^s_t],\bh^s_t}-\frac{\sigma-L_H}{6}\norm{\bh^s_t}^3\\
    &+\fprod{\bv^s_t,\bh^s_t}+\frac{1}{2}\fprod{\bU^s_t[\bh^s_t],\bh^s_t}+\frac{\sigma}{6}\norm{\bh^s_t}^3\\
    &\leq F(\bx_t^s)+(\frac{\sigma}{27}\norm{\bh^s_t}^3+\frac{2}{\sigma^{1/2}}\norm{\text{grad} F(\bx_t^s)-\bv_t^s}^{3/2})+\frac{1}{2}(\frac{2\sigma}{27}\norm{\bh^s_t}^3+\frac{27}{\sigma^2}\norm{H_{\bx_t^s}-\bU_t^s}_{op}^3)\\
    &-\frac{\sigma-L_H}{6}\norm{\bh^s_t}^3-\frac{\sigma}{12}\norm{\bh^s_t}^3\\
    &\leq F(\bx^s_t)+\frac{2}{\sigma^{1/2}}\norm{\text{grad} F(\bx_t^s)-\bv_t^s}^{3/2}+\frac{27}{2\sigma^2}\norm{H_{\bx^s_t}-\bU^s_t}_{op}^3-\frac{\sigma}{12}\norm{\bh^s_t}^3, \numberthis\label{eq:last_main_1}
\end{align*}
where the first inequality follows from Lemma~\ref{lemma:hessian} and the second inequality holds due to Lemmas~\ref{lemma:innerproduct} and~\ref{lemma:optimal}. Next, we define 
\begin{equation}\label{proof:11}
R^s_t=\mE[F(\bx^s_t)+c_t\norm{\text{Exp}^{-1}_{\hat{\bx}^s}(\bx^s_t)}^3],
\end{equation}
where $c_t$ is defined in Lemma~\ref{lemma:constant_c_t}. By Lemma~\ref{lemma:tri}, for $T\geq 2$, we have 
\begin{equation}\label{proof:22}
    c_{t+1}\norm{\text{Exp}^{-1}_{\hat{\bx}^s}(\text{Exp}_{\bx^s_t}(\bh^s_t))}^3 \leq 2c_{t+1}(\sqrt{\zeta-1}+1)^3 T^2\norm{\bh^s_t}^3+c_{t+1}(1+\frac{3}{T})\norm{\text{Exp}^{-1}_{\hat{\bx}^s}(\bx^s_t)}^3.
\end{equation}
From Lemma~\ref{lemma:mu_after_two_small} with $\bh=\bh^s_t$ using the condition \eqref{optimal:1} and the definition of $\bx^s_{t+1}$, we have
\begin{equation}\label{proof:33}
    \frac{\mu(\bx^s_{t+1})}{240 \bar{k}^2 \sqrt{L_H}}\leq \frac{\sigma}{24}\norm{\bh^s_t}^3+\frac{\norm{\text{grad} F(\bx^s_t)-\bv^s_t}^{3/2}}{24\sqrt{\sigma}}+\frac{\norm{\text{Hess} F(\bx^s_t)-\bU^s_t}^{3}}{24\sigma^2}.
\end{equation}
From \eqref{eq:last_main_1}, we have
\small
\begin{align*}
    &R^s_{t+1}+\mE [\frac{\mu(\bx^s_{t+1})}{240 \bar{k}^2 \sqrt{L_H}}]\\
    &=\mE[F(\bx^s_{t+1})+c_{t+1}\norm{\text{Exp}^{-1}_{\hat{\bx}^s}(\bx^s_{t+1})}^3+\frac{\mu(\bx^s_{t+1})}{240 \bar{k}^2 \sqrt{L_H}}]\\
    &\leq \mE[F(\bx^s_t)+\frac{3}{\sqrt{\sigma}}\norm{\text{grad} F(\bx^s_t)-\bv^s_t}^{3/2}+\frac{14}{\sigma^2}\norm{\text{Hess} F(\bx^s_t)-\bU^s_t}_{op}^{3}]\\
    &+\mE[c_{t+1}(1+\frac{3}{T})\norm{\text{Exp}^{-1}_{\hat{\bx}^s}(\bx^s_t)}^3-(\frac{\sigma}{24}-2c_{t+1}(\sqrt{\xi-1}+1)^3 T^2)\norm{\bh^s_t}^3]\\
    &\leq \mE[F(\bx^s_t)+\frac{3}{\sqrt{\sigma}}\norm{\text{grad} F(\bx^s_t)-\bv^s_t}^{3/2}+\frac{14}{\sigma^2}\norm{\text{Hess} F(\bx^s_t)-\bU^s_t}^{3}_{op}+c_{t+1}(1+\frac{3}{T})\norm{\text{Exp}^{-1}_{\hat{\bx}^s}(\bx^s_t)}^3],
\end{align*}
\normalsize
where the the first inequality follows from \eqref{eq:last_main_1}, \eqref{proof:22}, \eqref{proof:33} and the last inequality follows from Lemma~\ref{lemma:constant_c_t}.

Based on Lemma \ref{expectation1} and the conditions on $b_g$ and $b_h$, it can be verified that
\small
\begin{align*}
    \frac{3}{\sqrt{\sigma}}\mE\norm{\text{grad} F(\bx^s_t)-\bv^s_t}^{3/2}\leq \frac{3L_H^{3/2}}{\sqrt{\sigma}b_g^{3/4}}\mE \norm{\text{Exp}^{-1}_{\hat{\bx}^s}(\bx^s_t)}^3\leq \frac{\sigma}{1000T^3 (\sqrt{\zeta-1}+1)^3} \mE \norm{\text{Exp}^{-1}_{\hat{\bx}^s}(\bx^s_t)}^3,\\
    \frac{14}{\sigma^2}\mE\norm{\text{Hess} F(\bx^s_t)-\bU^s_t}^{3}_\leq \frac{896 L_H^3(\rho+\rho^2)^3}{\sigma^2} \mE\norm{\text{Exp}^{-1}_{\hat{\bx}^s}(\bx^s_t)}^3 \leq \frac{\sigma}{1000T^3 (\sqrt{\zeta-1}+1)^3} \mE \norm{\text{Exp}^{-1}_{\hat{\bx}^s}(\bx^s_t)}^3,
\end{align*}
where $\rho=\sqrt{\frac{2e\log mn}{b_h}}$. 
\normalsize
Therefore, we have
\begin{align*}
    R^s_{t+1}+\mE [\frac{\mu(\bx^s_{t+1})}{240 \bar{k}^2 \sqrt{L_H}}]&\leq \mE[F(\bx^s_t)+\norm{\text{Exp}^{-1}_{\hat{\bx}^s_t}(\bx^s_t)}^3(c_{t+1}(1+3/T)+\frac{\sigma}{500T^3(\sqrt{\zeta-1}+1)^3})]\\
    &=\mE[F(\bx^s_t)+c_t\norm{\text{Exp}^{-1}_{\hat{\bx}^s_t}(\bx^s_t)}^3]
    =R^s_t,
\end{align*}
where the first equality comes from the definition of $c_t$ in Lemma~\ref{lemma:constant_c_t}. Telescoping the above inequality from $t=0$ to $T-1$, we have 
\begin{equation*}
    R_0^s-R_T^s\geq (240\bar{k}^2 \sqrt{L_H})^{-1}\sum_{t=1}^T \mE[\mu(\bx^s_t)].
\end{equation*}
Note that $c_T=0$ and $\bx^{s-1}_T=\bx^s_0=\hat{\bx}^s$, then $R_T^s=\mE[F(\bx^s_T)+c_T\norm{\text{Exp}^{-1}_{\hat{\bx}^s_t}(\bx^s_T)}^3]=\mE F(\hat{\bx}^{s+1})$ and $R_0^s=\mE[F(\bx^s_0)+c_0\norm{\text{Exp}^{-1}_{\hat{\bx}^s}(\bx^s_0)}^3]=\mE F(\hat{\bx}^s)$, which implies 
\begin{equation*}
    \mE F(\hat{\bx}^s)-\mE F(\hat{\bx}^{s+1})=R_0^s-R_T^s\geq (240\bar{k}^2 \sqrt{L_H})^{-1}\sum_{t=1}^T \mE[\mu(\bx^s_t)].
\end{equation*}
Telescoping the above inequality from $s=1$ to $S$ yields 
\begin{equation*}
    \Delta_F \geq \sum^S_{s=1} \mE F(\hat{\bx}^s)-\mE F(\hat{\bx}^{s+1})\geq (240\bar{k}^2 \sqrt{L_H})^{-1}\sum_{s=1}^S \sum_{t=1}^T \mE[\mu(\bx^s_t)].
\end{equation*}
By the definition of the choice of $\bx_{out}$, the proof is completed. 
\end{proof}

%\begin{remark}
%Let $K=ST$ where $S$ and $T$ are the number of epochs and epoch length in Algorithm~\ref{alg:svr_cubic}. Then, by Theorem~\ref{main1}, $\mE(\mu(\bx)) \leq\epsilon$ after $O(1/K)$ iterations of the inner loop in Algorithm~\ref{alg:svr_cubic}, i.e., with $k=O(1/\epsilon)$, the algorithm obtains a solution with $\mE[\norm{\text{grad}F}^{3/2}]\leq\epsilon$ and $\mE[-\lambda^3_{min}(\text{Hess}F)]\leq L_H^{3/2}\epsilon$. Using reverse Jensen's inequality, after $K$ iterations, we have $\mE[\norm{\text{grad}F}]\leq\epsilon^{2/3}$ and $\mE[-\lambda_{\min}(\text{Hess}F)]\leq\epsilon^{1/3}$. 
%Hence, the algorithm converges to an $(\bar{\epsilon},\bar{\bar{\epsilon}})$-second order stationary point in $O(\max\{1/\bar{\epsilon}^{3/2},1/\bar{\bar{\epsilon}}^{3}\})$ iterations.
%%a $\bar{\epsilon}$-first-order stationary point in $O(1/\bar{\epsilon}^{3/2})$ iterations and to a $\bar{\bar{\epsilon}}$-second-order stationary point in $O(1/\bar{\bar{\epsilon}}^{3})$ iterations. 
%\end{remark}
\begin{remark}
Let $K=ST$ where $S$ and $T$ are the number of epochs and epoch length in Algorithm~\ref{alg:svr_cubic}. Following our discussion below \eqref{mux_expand} and by Theorem~\ref{main1}, setting $\mE[\mu(\bx)] \leq 240 \bar{k}^2 L_H^{1/2}\Delta_F/K\leq\epsilon^{3/2}$, the algorithm obtains a $(\epsilon,\sqrt{\epsilon})$-solution in $O(\epsilon^{-3/2})$ iterations. In other words, the algorithm obtains a first-order stationary point (i.e., $\norm{\text{grad} F(\bx)}\leq \epsilon$) in $O(\epsilon^{-3/2})$ iterations and a second-order stationary point (i.e, $\lambda_{\min} (\text{Hess} F(\bx))\geq -\epsilon$) in $O(\epsilon^{-3})$ iterations. 
\end{remark}

%\red{
%\begin{remark}
% after $O(1/K)$ iterations of the inner loop of Algorithm~\ref{alg:svr_cubic}, i.e., with $k=O(1/\epsilon)$, the algorithm obtains a solution with $\mE[\norm{\text{grad}F}^{3/2}]\leq\epsilon$ and $\mE[-\lambda^3_{min}(\text{Hess}F)]\leq L_H^{3/2}\epsilon$. Using reverse Jensen's inequality, after $K$ iterations, we have $\mE[\norm{\text{grad}F}]\leq\epsilon^{2/3}$ and $\mE[-\lambda_{\min}(\text{Hess}F)]\leq\epsilon^{1/3}$. 
%Hence, the algorithm converges to an $(\bar{\epsilon},\bar{\bar{\epsilon}})$-second order stationary point in $O(\max\{1/\bar{\epsilon}^{3/2},1/\bar{\bar{\epsilon}}^{3}\})$ iterations.
%\end{remark}
%}

\begin{definition}[Second-order oracle]\label{def:SO}
Given an index $i$ and a point $\bx$, a second-order oracle (SO) call returns a triple 
$
	[f_i(\bx),\nabla f_i(\bx), \nabla^2f_i(\bx)].
$
\end{definition}
When manifold is embedded in a Euclidean space, calculating the Riemannian gradient and Hessian (applied to a certain direction) requires the Euclidean gradient and Hessian. Therefore, the number of $SO$ calls is a reasonable metric to evaluate complexities of different algorithms, stochastic and deterministic. In numerical studies, we also compare different methods on the number of SO calls.

\begin{corollary}\label{corollary1}
Suppose that the cubic regularization parameter $\sigma$ in Algorithm~\ref{alg:svr_cubic} is fixed and satisfies $\sigma=\bar{k} L_H$, where $L_H$ is the Hessian Lipschitz parameter according to \eqref{lc0}, $\bar{k}\geq 2$ and $\sigma$ satisfies~\eqref{cond:sigma1}. Let the epoch length $T=N^{1/5}$, batch sizes $b_g=\frac{3000^{4/3}N^{4/5} (\sqrt{\zeta-1}+1)^4}{\bar{k}^2}$, $b_h=\frac{e\log d}{(\sqrt{\frac{\bar{k}}{193N^{1/5}(\sqrt{\zeta-1}+1)}+\frac{1}{8}}-\frac{1}{2\sqrt{2}})^2}$, and the number of epochs $S=\max\{1,240 \bar{k}^2 L_H^{1/2}\Delta_F N^{-1/5}\epsilon^{-3/2}\}$, where $d=mn$ is the dimension of the problem. Then, under Assumptions~\ref{assumption:smooth}-\ref{assumption:curvature}, Algorithm~\ref{alg:svr_cubic} finds an $(\epsilon, \sqrt{L_H\epsilon})$-second-order stationary point in $\tilde{O}(N+L_H^{1/2}\Delta_F N^{4/5}\epsilon^{-3/2})$ second-order oracle calls.
\end{corollary}

\begin{proof}
	The parameter setting in Corollary~\ref{corollary1} satisfies the requirements of Theorem~\ref{main1}. The epoch size $S$ enforce $\mE[\mu(\bx_{out})]\leq \epsilon$, which implies that $\bx_{out}$ is an $(\epsilon, \sqrt{L_H\epsilon})$-approximate local minimum. 
	Note that Algorithm~\ref{alg:svr_cubic} requires calculating full gradient $\nabla F$ and Hessian $\nabla^2 F$ at the beginning of each epoch with $N$ SO calls. Inside each epoch, it needs to calculate stochastic gradient and Hessian with $b_g+b_h$ SO calls at each iteration. Thus, the total number of SO calls is 
	\begin{align*}
		SN+(ST)(b_g+b_h)&\leq N+240 \bar{k}^2 L_H^{1/2}\Delta_F N^{4/5}\epsilon^{-3/2}
		+240 \bar{k}^2 L_H^{1/2}\Delta_F \epsilon^{-3/2}(b_g+b_h)\\
		&=\tilde{O}(N+L_H^{1/2}\Delta_F N^{4/5}\epsilon^{-3/2}),
	\end{align*}
where the $\tilde{O}$ comes from $\log d$ in $b_h$.
\end{proof}

In practice, finding the exact solution to the cubic-regularized Newton subproblem \eqref{m_cubic} is not always computationally desirable~\cite{agarwal2020adaptive,nester2006cubic,cartis2011adaptive,cartis2011adaptive2}. Instead, we can solve the subproblem \emph{inexactly}, but yet guarantee theoretical properties of the algorithm. More specifically, we propose to solve the cubic-regularized Newton subproblem inexactly, but the one that satisfies the conditions in Definition~\ref{def:inexact_sol} below. It is then proved in Theorem~\ref{main2} that the complexity of the algorithm with inexact solution to its subproblem is the same as the original algorithm, except for an $O(1)$ constant.

\begin{definition}[Inexact solution] \label{def:inexact_sol}
Given a $\delta>0$, $\tilde{\bh}^s_t$ is a $\delta$-inexact solution to \eqref{m_cubic} if it satisfies 
\begin{align}
    m^s_t(\tilde{\bh}^s_t)&\leq -\frac{\sigma}{12}\norm{\tilde{\bh}^s_t}^3+\delta, \label{def:1}\\
    \norm{\nabla m^s_t(\tilde{\bh}^s_t)}&\leq (\sigma)^{1/3}\delta^{2/3}, \label{def:2}\\ 
    \lambda_{\min} (\nabla^2 m^s_t(\tilde{\bh}^s_t)) &\geq -(\sigma)^{2/3}\delta^{1/3}. \label{def:3}
\end{align}
\end{definition}

The following lemma is parallel to Lemma~\ref{lemma:optimal} when the subproblem is solved inexactly.
\begin{lemma}\label{lemma2:optimal}
Under Assumption~\ref{assumption:embedded}, if $\tilde{\bh}^s_t$ is a $\delta$-inexact solution to \eqref{m_cubic}, then 
\begin{align}
    \fprod{\bv^s_t,\tilde{\bh}^s_t}+\frac{1}{2}\fprod{\bU^s_t\tilde{\bh}^s_t,\tilde{\bh}^s_t}+\frac{\sigma}{6}\norm{\tilde{\bh}^s_t}^3&\leq-\frac{\sigma}{12}\norm{\tilde{\bh}^s_t}^3+\delta, \label{optimal2:1}\\
     \norm{\bv^s_t+\bU^s_t\tilde{\bh}^s_t+(\frac{\sigma}{2}\norm{\tilde{\bh}^s_t})\tilde{\bh}^s_t}&\leq (\sigma)^{1/3}\delta^{2/3}, \label{optimal2:2}\\
    \bU^s_t+\sigma\norm{\tilde{\bh}^s_t}\bI&\succeq -(\sigma)^{2/3}\delta^{1/3}\bI. \label{optimal2:3}
\end{align}
\end{lemma}

\begin{proof}
Inequalities~\eqref{optimal2:1} and \eqref{optimal2:2} follow from expanding $m^s_t(\tilde{\bh}^s_t)$ and $\nabla m^s_t(\tilde{\bh}^s_t)$ in \eqref{def:1} and \eqref{def:2}.
To show \eqref{optimal2:3}, note that $\nabla^2 m^s_t(\tilde{\bh}^s_t)=\bU^s_t+\lambda \bI+\lambda(\frac{\tilde{\bh}^s_t}{\norm{\tilde{\bh}^s_t}})(\frac{\tilde{\bh}^s_t}{\norm{\tilde{\bh}^s_t}})^\top$, where $\lambda=\frac{\sigma \norm{\tilde{\bh}^s_t}}{2}$. We have
\begin{align*}
    \bU^s_t+2\lambda \bI \succeq \bU^s_t+\lambda \bI+\lambda(\frac{\tilde{\bh}^s_t}{\norm{\tilde{\bh}^s_t}})(\frac{\tilde{\bh}^s_t}{\norm{\tilde{\bh}^s_t}})^\top
     \succeq-(\sigma)^{2/3}\delta^{1/3}\bI,
\end{align*}
where the first inequality follows from the Cauchy--Schwarz inequality, $\norm{\bv}\geq \frac{\fprod{\bv, \tilde{\bh}^s_t}}{\norm{\tilde{\bh}^s_t}}$ for any $\bv\in \mR^{m\times n}$, and the second inequality follows from \eqref{optimal2:3}.
\end{proof}

Parallel to Lemma~\ref{lemma:bound_stepsize_via_sigma}, Lemma~\ref{lemma:bound_stepsize_via_sigma2} provides an upper bound on $\norm{\tilde{\bh}_s^t}$ which then provides a required (lower) bound on the cubic regularization parameter $\sigma$ to have the iterates close enough to the epoch points - see Remark~\ref{remark:control_stepsize}.
\begin{lemma}\label{lemma:bound_stepsize_via_sigma2}
	Under Assumption~\ref{assumption:smooth}-\ref{assumption:positiveinj}, given a constant $C>0$ and $\sigma>[\frac{\delta^{2/3}+\sqrt{\delta^{4/3}+2C^2\cdot(C\cdot c_H+c_g)}}{C^2}]^2$, we have $\norm{\tilde{\bh}^s_t}< C$. 
\end{lemma}
\begin{proof}
	Based on~\eqref{optimal2:2} and Cauchy--Schwarz inequality, we have 
	\begin{equation}
		\fprod{\bv^s_t,\tilde{\bh}^s_t}+\fprod{\bU^s_t\tilde{\bh}^s_t,\tilde{\bh}^s_t}+\frac{\sigma}{2}\norm{\tilde{\bh}^s_t}^3\leq (\sigma)^{1/3}\delta^{2/3}\cdot\norm{\tilde{\bh}^s_t}. 
	\end{equation}
	which implies, 
	\begin{equation}
		\frac{\sigma}{2}\norm{\tilde{\bh}^s_t}^3\leq \norm{\bv^s_t}\cdot\norm{\tilde{\bh}^s_t}+\norm{\bU^s_t}_{op}\cdot\norm{\tilde{\bh}^s_t}^2+(\sigma)^{1/3}\delta^{2/3}\cdot\norm{\tilde{\bh}^s_t}.
	\end{equation}
	Based on~\eqref{assumptionremark:bound_hessian_grad} and dividing both sides by $\norm{\tilde{\bh}^s_t}$, we have 
	\begin{equation}
		\frac{\sigma}{2}\norm{\tilde{\bh}^s_t}^2-c_H\cdot \norm{\tilde{\bh}^s_t}-(+\sigma^{1/3}\delta^{2/3})\leq 0,
	\end{equation}
	which implies 
	\begin{align}
		\norm{\tilde{\bh}^s_t}&\leq\frac{c_H+\sqrt{c_H^2+2\sigma (c_g+\sigma^{1/3}\delta^{2/3})}}{\sigma}
		\leq\frac{c_H+\sqrt{c_H^2+2\sigma (c_g+\sigma^{1/2}\delta^{2/3})}}{\sigma}.\label{lemmaproof:2}
	\end{align}
	Note that the right hand side of~\eqref{lemmaproof:2} is a monotonic decreasing function on $\sigma$. After some simple manipulation, we derive that if $\sigma>[\frac{\delta^{2/3}+\sqrt{\delta^{4/3}+2C^2\cdot(C\cdot c_H+c_g)}}{C^2}]^2$, then the right hand side of~\eqref{lemmaproof:2} is upper bounded by $C$, which implies $\norm{\tilde{\bh}^s_t} <C$. 
\end{proof}

\begin{remark}\label{remark:control_stepsize}
	Using Lemma~\ref{lemma:bound_stepsize_via_sigma}, setting $C=\text{inj}(\cM)/T$, where $T$ is the epoch length of the algorithm, we have
	\begin{equation}\label{cond:sigma2}
		\sigma>[\frac{T^2\delta^{2/3}+\sqrt{T^4\delta^{4/3}+2(\text{inj}(\cM))^2(\text{inj}(\cM)Tc_H+c_g T^2)}}{(\text{inj}(\cM))^2}]^2.
	\end{equation}
	Given the lower bound on $\sigma$, for any epoch $s$ and any iteration $t$ inside this epoch, we have
	\begin{equation}
		d(\hat{\bx}^s,\bx^s_t)\leq \sum_{i=1}^{t} d(\bx^s_{i-1},\bx^s_i) \leq \sum_{i=1}^{t}\norm{\bh^s_i}<\text{inj}(\cM). 
	\end{equation}
\end{remark}

Since it is difficult to quantify the difference of $\bh$ with the exact solution $\bh^s_t$, i.e. $|\norm{\bh}-\norm{\bh^s_t}|$, we need to establish results similar to Lemmas~\ref{lemma:small2} and \ref{lemma:mu_after_two_small} based on \eqref{optimal2:3}.

\begin{lemma}\label{lemma:small3}
Let $\tilde{\bh}^s_t$ be a $\delta$-inexact solution to \eqref{m_cubic} with $\sigma\geq 2 L_H$ that satisfies~\eqref{cond:sigma2}, then under Assumptions~\ref{assumption:smooth}-\ref{assumption:positiveinj}, for any $\bh\in\mR^{m\times n}$ such that $\norm{\bh}<\text{inj}(\cM)$, we have
  \begin{equation}
      -\lambda_{\min} (\text{Hess} F(\text{Exp}_{\bx^s_t}(\bh)) \leq \frac{3\sigma}{2}\norm{\bh}+\norm{H_{\bx}-U^s_t}_{op}+\sigma |\norm{\tilde{\bh}^s_t}-\norm{\bh}|+(\sigma)^{2/3}\delta^{1/3},
  \end{equation}
where $\lambda_{\min} (\text{Hess} F(\bx))$ is defined as $\lambda_{\min} (\text{Hess} F(\bx))\triangleq\inf_{\eta\in T_{\bx}\cM} \{\frac{\fprod{\text{Hess} F(\bx)\eta,\eta}}{\norm{\eta}}\}$.
\end{lemma}
\begin{proof}
Denote $\text{Exp}_{\bx^s_t}(\bh)$ by $\by$ and $\bx^s_t$ by $\bx$.  Furthermore, let the identity operator at $\bx$ be denoted by $\mathbf{I}_{\bx}$, i.e. $\mathbf{I}_{\bx}(\eta)=\eta$ for any $\eta\in T_{\bx} \cM$. 
  We have 
\begin{align*}
    H_{\by} &\succeq \Gamma_{\bx}^{\by} H_{\bx} \Gamma_{\by}^{\bx}-L_H \norm{\bh} \mathbf{I}_{\by}\\
    & \succeq \Gamma_{\bx}^{\by} U^s_t \Gamma_{\by}^{\bx}-\norm{\Gamma_{\bx}^{\by} H_{\bx} \Gamma_{\by}^{\bx}-\Gamma_{\bx}^{\by} U^s_t \Gamma_{\by}^{\bx}}_{op} \mathbf{I}_{\by}-L_H \norm{\bh} \mathbf{I}_{\by}\\
    & \succeq -(\sigma\norm{\tilde{\bh}^s_t}+(\sigma)^{2/3}\delta^{1/3}) \mathbf{I}_{\by}-\norm{H_{\bx}-U^s_t}_{op} \mathbf{I}_{\by}-L_H\norm{\bh} \mathbf{I}_{\by},
\end{align*}
where the first inequality follows from \eqref{lc0}, the second inequality follows from the definition of the operator norm and the triangle inequality, and the third inequality follows from the isometric property of the parallel transport $\Gamma$ and \eqref{optimal2:3}.
Therefore, we have 
\begin{align*}
    -\lambda_{\min} (H_{\by}) &\leq (\sigma\norm{\tilde{\bh}^s_t}+(\sigma)^{2/3}\delta^{1/3})+\norm{H_{\bx}-U^s_t}_{op}+ L_H\norm{\bh}\\
    &=\sigma (\norm{\tilde{\bh}^s_t}-\norm{\bh})+\norm{H_{\bx}-U^s_t}_{op}+ (L_H+\sigma)\norm{\bh}+(\sigma)^{2/3}\delta^{1/3}\\
    &\leq \frac{3\sigma}{2}\norm{\bh}+\norm{H_{\bx}-U^s_t}_{op}+\sigma |\norm{\tilde{\bh}^s_t}-\norm{\bh}|+(\sigma)^{2/3}\delta^{1/3},
\end{align*}
where the last inequality holds because $L_H \leq \sigma/2$. 
\end{proof}

Recall the $\mu(\bx)$ definition in \eqref{mux}, combining Lemmas \ref{lemma:small1} and \ref{lemma:small3}, we have the following result.
\begin{lemma}\label{lemma2:mu_after_two_small}
Setting $\sigma=\bar{k} L_H$ with $\bar{k}\geq 2$ such that it satisfies~\eqref{cond:sigma2}, under Assumptions~\ref{assumption:smooth}-\ref{assumption:positiveinj}, for any $\delta$-inexact solution $\tilde{\bh}^s_t$,  we have  
\begin{align*}
     \mu (\text{Exp}_{\bx^s_t}(\tilde{\bh}^s_t))&\leq 9(\bar{k})^{3/2}[\frac{27(\sigma)^{3/2}}{8}\norm{\tilde{\bh}^s_t}^3+\norm{\text{grad} F(\bx^s_t)-\bv^s_t}^{3/2}
     +(\sigma)^{-3/2}\norm{\text{Hess}F(\bx^s_t)-\bU^s_t}^3_{op}+(\sigma)^{1/2}\delta].
\end{align*}
 \end{lemma}
\begin{proof}
 Recall $\mu (\bx) =\max \{\norm{\text{grad} F(\bx)}^{3/2}, -L_H^{-3/2}\lambda^3_{min}(\text{Hess}F(\bx))\}$. Next, we apply Lemmas \ref{lemma:small1} and \ref{lemma:small3} to upper bound $\norm{\text{grad} F(\bx)}^{3/2}$ and $-(L_H^{3/2})^{-1}[\lambda_{\min}(\text{Hess}F(\bx))]^3\}$, respectively. 
 \begin{align*}
     &\norm{\text{grad} F(\text{Exp}_{\bx^s_t}(\tilde{\bh}^s_t))}^{3/2}\\
     &\leq [\sigma\norm{\tilde{\bh}^s_t}^2+\norm{\text{grad} F(\bx^s_t)-\bv^s_t}+\frac{1}{\sigma}\norm{\text{Hess} F(\bx^s_t)-\bU^s_t}^2_{op}+\norm{\nabla m^s_t(\tilde{\bh}^s_t)}]^{3/2}\\
     &\leq 2[(\sigma)^{3/2}\norm{\tilde{\bh}^s_t}^3+\norm{\text{grad} F(\bx^s_t)-\bv^s_t}^{3/2}+(\sigma)^{-3/2}\norm{\text{Hess} F(\bx^s_t)-\bU^s_t}^3_{op}+\norm{\nabla m^s_t(\tilde{\bh}^s_t)}^{3/2}]\\
     &\leq 2[(\sigma)^{3/2}\norm{\tilde{\bh}^s_t}^3+\norm{\text{grad} F(\bx^s_t)-\bv^s_t}^{3/2}+(\sigma)^{-3/2}\norm{\text{Hess} F(\bx^s_t)-\bU^s_t}^3_{op}+(\sigma)^{1/2}\delta],
 \end{align*}
 where the first inequality follows from Lemma \ref{lemma:small1}, the second inequality holds due to the inequality $(a+b+c+d)^{3/2}\leq 2(a^{3/2}+b^{3/2}+c^{3/2}+d^{3/2})$, and the third inequality follows from \eqref{def:2}.
 \begin{align*}
     -L_H^{-3/2}[\lambda_{\min}(\text{Hess}F(\tilde{\bh}^s_t))]^3
     &=-(\bar{k})^{3/2} (\sigma)^{-3/2}[\lambda_{\min}(\text{Hess}F(\tilde{\bh}^s_t))]^3\\
     &\leq (\bar{k})^{3/2} (\sigma)^{-3/2}[\frac{3\sigma}{2}\norm{\tilde{\bh}^s_t}+\norm{\text{Hess} F(\bx^s_t)-\bU^s_t}_{op}+(\sigma)^{2/3}\delta^{1/3}]^3\\
     &\leq 9(\bar{k})^{3/2}[\frac{27(\sigma)^{3/2}}{8}\norm{\tilde{\bh}^s_t}^3+(\sigma)^{-3/2}\norm{\text{Hess}F(\bx^s_t)-\bU^s_t}^3_{op}+(\sigma)^{1/2}\delta],
 \end{align*}
 where the equality follows from $\sigma=\bar{k} L_H$, the first inequality follows from Lemma \ref{lemma:small3}, and the last inequality follows from the inequality $(a+b+c)^3\leq 9(a^3+b^3+c^3)$. Since $9(\bar{k})^{3/2}>2$, we have 
 \begin{align*} \small
     \mu (\text{Exp}_{\bx^s_t}(\tilde{\bh}^s_t)) &=\max \{\norm{\text{grad} F(\text{Exp}_{\bx^s_t}(\tilde{\bh}^s_t))}^{3/2}, -L_H^{-3/2}\lambda^3_{min}(\text{Hess}F(\text{Exp}_{\bx^s_t}(\tilde{\bh}^s_t))\}\\
     &\leq 9(\bar{k})^{3/2}[\frac{27(\sigma)^{3/2}}{8}\norm{\tilde{\bh}^s_t}^3+\norm{\text{grad} F(\bx^s_t)-\bv^s_t}^{3/2}
     +(\sigma)^{-3/2}\norm{\text{Hess}F(\bx^s_t)-\bU^s_t}^3_{op}+(\sigma)^{1/2}\delta],
 \end{align*}
which completes the proof. 
 \end{proof}

Theorem~\ref{main2} below provides the convergence rate of the R-SVRC algorithm when the cubic regularized Newton subproblem is solved \emph{inexactly}.
\begin{theorem}\label{main2}
Suppose that the cubic regularization parameter $\sigma$ in Algorithm~\ref{alg:svr_cubic} is fixed and satisfies $\sigma=\bar{k} L_H$, where $L_H$ is the Hessian Lipschitz parameter according to \eqref{lc0} and $\bar{k}\geq 2$ and it also satisfies~\eqref{cond:sigma2}. At each iteration, let the cubic subproblem \eqref{m_cubic} be solved inexactly so that the results $\{\tilde{\bh}^s_t\}$ are $\delta$-inexact solutions. Furthermore, suppose that the batch sizes $b_g$ and $b_h$ satisfy 
\begin{equation}\label{condition:batch_sizes}
    b_g \geq \frac{3000^{4/3}T^4 (\sqrt{\xi-1}+1)^4}{\bar{k}^2}, \quad b_h \geq \frac{e\log d}{(\sqrt{\frac{\bar{k}}{193T(\sqrt{\xi-1}+1)}+\frac{1}{8}}-\frac{1}{2\sqrt{2}})^2}, 
\end{equation}
where $T\geq 2$ is the length of the inner loop of the algorithm and $d=mn$ is the dimension of the problem. Then, under Assumptions~\ref{assumption:smooth}-\ref{assumption:curvature}, the output of the algorithm satisfies 
\begin{equation}
    \mE [\mu(\bx_{out})]\leq \frac{729 \bar{k}^2 L_H^{1/2}\Delta_F}{S T}+738\bar{k}^2 \sqrt{L_H}\delta,
\end{equation}
where $\mu(\bx)$ is defined in \eqref{mux}.
\end{theorem}

\begin{proof}
First, we upper bound $F(\bx_{t+1}^s)$ as follows:
\begin{align*}
    F(\bx_{t+1}^s)&\leq F(\bx_t^s)+\fprod{\text{grad} F(\bx_t^s),\tilde{\bh}^s_t}+\frac{1}{2}\fprod{H_{\bx_t^s}[\tilde{\bh}^s_t],\tilde{\bh}^s_t}+\frac{L_H}{6}\norm{\tilde{\bh}^s_t}^3\\
    &=F(\bx_t^s)+\fprod{\text{grad} F(\bx_t^s)-\bv_t^s,\tilde{\bh}^s_t}+\frac{1}{2}\fprod{(H_{\bx_t^s}-\bU_t^s)[\tilde{\bh}^s_t],\tilde{\bh}^s_t}-\frac{\sigma-L_H}{6}\norm{\tilde{\bh}^s_t}^3\\
    &+\fprod{\bv^s_t,\tilde{\bh}^s_t}+\frac{1}{2}\fprod{\bU^s_t[\tilde{\bh}^s_t],\tilde{\bh}^s_t}+\frac{\sigma}{6}\norm{\tilde{\bh}^s_t}^3\\
    &\leq F(\bx_t^s)+(\frac{\sigma}{27}\norm{\tilde{\bh}^s_t}^3+\frac{2}{\sigma^{1/2}}\norm{\text{grad} F(\bx_t^s)-\bv_t^s}^{3/2})+\frac{1}{2}(\frac{2\sigma}{27}\norm{\tilde{\bh}^s_t}^3+\frac{27}{\sigma^2}\norm{H_{\bx_t^s}-\bU_t^s}_{op}^3)\\
    &-\frac{\sigma-L_H}{6}\norm{\tilde{\bh}^s_t}^3-\frac{\sigma}{12}\norm{\tilde{\bh}^s_t}^3+\delta\\
    &\leq F(\bx^s_t)+\frac{2}{\sigma^{1/2}}\norm{\text{grad} F(\bx_t^s)-\bv_t^s}^{3/2}+\frac{27}{2\sigma^2}\norm{H_{\bx^s_t}-\bU^s_t}_{op}^3-\frac{\sigma}{12}\norm{\tilde{\bh}^s_t}^3+\delta,
\end{align*}
where the first inequality follows from H-smooth assumption and Lemma~\ref{lemma:hessian}, and the second inequality holds due to Lemma~\ref{lemma:innerproduct} and \eqref{optimal2:1} in Lemma~\ref{lemma2:optimal}.

Next, we define 
\begin{equation}\label{proof:1}
R^s_t\triangleq\mE[F(\bx^s_t)+c_t\norm{\text{Exp}^{-1}_{\hat{\bx}^s}(\bx^s_t)}^3],
\end{equation}
where $c_t$ is defined in Lemma~\ref{lemma:constant_c_t}. By Lemma~\ref{lemma:tri}, for $T\geq 2$, we have 
\begin{equation}\label{proof:2}
    c_{t+1}\norm{\text{Exp}^{-1}_{\hat{\bx}^s}(\text{Exp}_{\bx^s_t}(\tilde{\bh}^s_t))}^3 \leq 2c_{t+1}(\sqrt{\xi-1}+1)^3 T^2\norm{\tilde{\bh}^s_t}^3+c_{t+1}(1+\frac{3}{T})\norm{\text{Exp}^{-1}_{\hat{\bx}^s}(\bx^s_t)}^3.
\end{equation}
Furthermore, from Lemma~\ref{lemma2:mu_after_two_small}, we have
\begin{equation}\label{proof:3}
    \frac{\mu(\bx^s_{t+1})}{729 \bar{k}^2 \sqrt{L_H}}\leq \frac{\sigma}{24}\norm{\tilde{\bh}^s_t}^3+\frac{\norm{\text{grad} F(\bx^s_t)-\bv^s_t}^{3/2}}{81\sqrt{\sigma}}+\frac{\norm{\text{Hess} F(\bx^s_t)-\bU^s_t}^{3}}{81\sigma^2}+\frac{\delta}{81}.
\end{equation}
Combining \eqref{proof:1}, \eqref{proof:2} and \eqref{proof:3}, we have 
\small
\begin{align*}
    &R^s_{t+1}+\mE [\frac{\mu(\bx^s_{t+1})}{729 \bar{k}^2 \sqrt{L_H}}]\\
    &=\mE[F(\bx^s_{t+1})+c_{t+1}\norm{\text{Exp}^{-1}_{\hat{\bx}^s}(\bx^s_{t+1})}^3+\frac{\mu(\bx^s_{t+1})}{729 \bar{k}^2 \sqrt{L_H}}]\\
    &\leq \mE[F(\bx^s_t)+\frac{3}{\sqrt{\sigma}}\norm{\text{grad} F(\bx^s_t)-\bv^s_t}^{3/2}+\frac{14}{\sigma^2}\norm{\text{Hess} F(\bx^s_t)-\bU^s_t}^{3}]\\
    &+\mE[c_{t+1}(1+\frac{3}{T})\norm{\text{Exp}^{-1}_{\hat{\bx}^s}(\bx^s_t)}^3-(\frac{\sigma}{24}-2c_{t+1}(\sqrt{\xi-1}+1)^3 T^2)\norm{\tilde{\bh}^s_t}^3]+\frac{82\delta}{81}\\
    &\leq \mE[F(\bx^s_t)+\frac{3}{\sqrt{\sigma}}\norm{\text{grad} F(\bx^s_t)-\bv^s_t}^{3/2}+\frac{14}{\sigma^2}\norm{\text{Hess} F(\bx^s_t)-\bU^s_t}^{3}_{op}
    +c_{t+1}(1+\frac{3}{T})\norm{\text{Exp}^{-1}_{\hat{\bx}^s}(\bx^s_t)}^3]+\frac{82\delta}{81},
\end{align*}
\normalsize
where the last inequality follows from Lemma~\ref{lemma:constant_c_t}.

Based on Lemma~\ref{expectation1} and the conditions on the sizes of $b_g$ and $b_h$, we have
\small
\begin{align}
    \frac{3}{\sqrt{\sigma}}\mE\norm{\text{grad} F(\bx^s_t)-\bv^s_t}^{3/2}\leq \frac{3L_H^{3/2}}{\sqrt{\sigma}b_g^{3/4}}\mE \norm{\text{Exp}^{-1}_{\hat{\bx}^s}(\bx^s_t)}^3\leq \frac{\sigma}{1000T^3 (\sqrt{\xi-1}+1)^3} \mE \norm{\text{Exp}^{-1}_{\hat{\bx}^s}(\bx^s_t)}^3,\\
    \frac{14}{\sigma^2}\mE\norm{\text{Hess} F(\bx^s_t)-\bU^s_t}^{3}_\leq \frac{896 L_H^3(\rho+\rho^2)^3}{\sigma^2} \mE\norm{\text{Exp}^{-1}_{\hat{\bx}^s}(\bx^s_t)}^3 \leq \frac{\sigma}{1000T^3 (\sqrt{\xi-1}+1)^3} \mE \norm{\text{Exp}^{-1}_{\hat{\bx}^s}(\bx^s_t)}^3,
\end{align}
\normalsize
where $\rho=\sqrt{\frac{2e\log mn}{b_h}}$.
Therefore, we have 
\small
\begin{align*}
    R^s_{t+1}+\mE [\frac{\mu(\bx^s_{t+1})}{729 \bar{k}^2 \sqrt{L_H}}]&\leq \mE[F(\bx^s_t)+\norm{\text{Exp}^{-1}_{\hat{\bx}^s_t}(\bx^s_t)}^3(c_{t+1}(1+3/T)+\frac{\sigma}{500T^3(\sqrt{\xi-1}+1)^3})]+\frac{82\delta}{81}\\
    &=\mE[F(\bx^s_t)+c_t\norm{\text{Exp}^{-1}_{\hat{\bx}^s_t}(\bx^s_t)}^3]+\frac{82\delta}{81}
    =R^s_t+\frac{82\delta}{81},
\end{align*}
\normalsize
where the first equality is due to the choice of $\{c_t\}$ defined in Lemma~\ref{lemma:constant_c_t}. Telescoping the above inequality from $t=0$ to $T-1$, we have 
$
    R_0^s-R_T^s\geq (729\bar{k}^2 \sqrt{L_H})^{-1}\sum_{t=1}^T (\mE[\mu(\bx^s_t)]-\frac{82\delta}{81}).
$
Note that $c_T=0$ and $\bx^{s-1}_T=\bx^s_0=\hat{\bx}^s$, then $R_T^s=\mE[F(\bx^s_T)+c_T\norm{\text{Exp}^{-1}_{\hat{\bx}^s_t}(\bx^s_T)}^3]=\mE F(\hat{\bx}^{s+1})$ and $R_0^s=\mE[F(\bx^s_0)+c_0\norm{\text{Exp}^{-1}_{\hat{\bx}^s}(\bx^s_0)}^3]=\mE F(\hat{\bx}^s)$, which implies 
\begin{equation*}
    \mE F(\hat{\bx}^s)-\mE F(\hat{\bx}^{s+1})=R_0^s-R_T^s\geq (729 \bar{k}^2 \sqrt{L_H})^{-1}\sum_{t=1}^T (\mE[\mu(\bx^s_t)]-\frac{82\delta}{81}).
\end{equation*}
Telescoping the above inequality from $s=1$ to $S$ yields 
\begin{equation*}
    \Delta_F \geq \sum^S_{s=1} \mE F(\hat{\bx}^s)-\mE F(\hat{\bx}^{s+1})\geq (729\bar{k}^2 \sqrt{L_H})^{-1}\sum_{s=1}^S \sum_{t=1}^T (\mE[\mu(\bx^s_t)]-\frac{82\delta}{81}).
\end{equation*}
By the definition of $\bx_{out}$, the proof is completed. 
\end{proof}

\begin{corollary}\label{cor:inexact}
For any $s$ and $t$, let $\tilde{\bh}^s_t$ be an inexact solution of the cubic subproblem $m^s_t(\bh)$, which satisfies Definition~\ref{def:inexact_sol} with $\delta=(1500\bar{k}^2\sqrt{L_H})^{-1}\epsilon^{3/2}$. Suppose that the cubic regularization parameter $\sigma$ in Algorithm~\ref{alg:svr_cubic} is fixed and satisfies $\sigma=\bar{k} L_H$, where $L_H$ is the Hessian Lipschitz parameter according to \eqref{lc0} with $\bar{k}\geq 2$, and it also satisfies~\eqref{cond:sigma2}. Let the epoch length $T=N^{1/5}$, batch sizes $b_g=\frac{3000^{4/3}N^{4/5} (\sqrt{\zeta-1}+1)^4}{\bar{k}^2}$, $b_h=\frac{e\log d}{(\sqrt{\frac{\bar{k}}{193N^{1/5}(\sqrt{\zeta-1}+1)}+\frac{1}{8}}-\frac{1}{2\sqrt{2}})^2}$, and the number of epochs $S=\max\{1,1500 \bar{k}^2 L_H^{1/2}\Delta_F N^{-1/5}\epsilon^{-3/2}\}$. Then, under Assumptions~\ref{assumption:smooth}-\ref{assumption:curvature}, Algorithm~\ref{alg:svr_cubic} finds an $(\epsilon, \sqrt{L_H\epsilon})$-second-order stationary point in $\tilde{O}(N+L_H^{1/2}\Delta_F N^{4/5}\epsilon^{-3/2})$ number of second-order oracle calls.	
\end{corollary}
\begin{proof}
	Under the parameter setting in Corollary~\ref{cor:inexact} and Theorem~\ref{main2}, we have
	\begin{equation}
		\mE[\mu(\bx_{out})]\leq \frac{729 \bar{k}^2 L_H^{1/2}\Delta_F}{S T}+738\bar{k}^2 \sqrt{L_H}\delta\leq\frac{\epsilon^{3/2}}{2}+\frac{\epsilon^{3/2}}{2}=\epsilon^{3/2}.
	\end{equation}
	Thus, $\bx_{out}$ is an $(\epsilon, \sqrt{L_H\epsilon})$-approximate local minimum. Similar to the discussion in Corollary~\ref{corollary1}, the total number of SO calls is 
	\begin{align*}
	SN+(ST)(b_g+b_h)&\leq N+1500 \bar{k}^2 L_H^{1/2}\Delta_F N^{4/5}\epsilon^{-3/2}
	+1500 \bar{k}^2 L_H^{1/2}\Delta_F \epsilon^{-3/2}(b_g+b_h)\\
	&=\tilde{O}(N+L_H^{1/2}\Delta_F N^{4/5}\epsilon^{-3/2}).
	\end{align*}
\end{proof}

%\vspace{-0.75cm}
\section{Numerical Studies}
In this section, we conduct numerical experiments to verify our theoretical complexity results for the R-SVRC algorithm to find a second-order stationary point. Besides different simulation studies, we compare our algorithm with crude Riemannian cubic regularization Newton method (CRC), Riemannian adaptive cubic regularization method (ARC), and Riemannian trust region method (RTR) -- see~\cite{zhang2018cubic,agarwal2020adaptive,absil2007trust}. Our code is written in conformance with the Manopt package~\cite{manopt}, and it is available at \href{https://github.com/samdavanloo/R-SVRC}{https://github.com/samdavanloo/R-SVRC}. All the numerical studies are run on a laptop with 1.4 GHz Quad-Core Intel Core i5 CPU and 8 GB memory.

\subsection{Parameter Estimation of Multivariate Student's t-distribution}
The maximum likelihood estimation of the (scale) parameter of the multivariate t-distribution~\eqref{student_t} requires solving
\begin{equation}\label{numerical:student_distribution}
	\min_{X\in\cS^p_{++}} F(X)=\frac{\nu+p}{2N}\sum_{i=1}^N \log(1+\frac{\ba_i^T X\ba_i}{\nu})-\frac{1}{2}\log\det(X),
\end{equation}
where the mean is assumed to be zero and $X$ is the inverse of the scale matrix $\Sigma$ which should belong to the Symmetric Positive Definite (SPD) manifold. The Euclidean gradient and Hessian of $F$ can be calculated as
\begin{align}
	\nabla F(X)&=\frac{\nu+p}{2N}\sum_{i=1}^N \frac{\ba_i\ba_i^T}{\nu+\ba_i^T X\ba_i}-\frac{1}{2}X^{-1},\label{numerical:gradient}\\ 
	\nabla^2 F(X)[U]&=\frac{\nu+p}{2N}\sum_{i=1}^N\frac{-\ba_i^T U \ba_i}{(\nu+\ba_i^T X\ba_i)^2}\ba_i\ba_i^T+\frac{1}{2}X^{-1}UX^{-1}\\
	&=[\frac{\nu+p}{2N}\sum_{i=1}^N\frac{-(\ba_i\ba_i^T)\otimes(\ba_i\ba_i^T)}{(\nu+\ba_i^T X\ba_i)^2}+\frac{1}{2}X^{-1}\otimes X^{-1}]\cdot \text{vec}(U),\label{numerical:tensor}
\end{align}
where $\text{sym}(Y)\triangleq\frac{1}{2}(Y+Y^\top)$, $\text{vec}(\cdot)$ denotes vectorization of the input matrix, and $\otimes$ denotes the Kronecker product.
The Riemannian gradient and Hessian are obtained as (see~\cite{bhatia2009positive,manopt}):
\begin{align}
	\text{grad} F(X)&=X\text{sym}(\nabla F(X))X,\\
	\text{Hess} F(X)[U]&=X\text{sym}(\nabla^2 F(X)[U])X+\text{sym}(U\nabla F(X)X).
\end{align}
While the above equations compute the full gradient and Hessian along certain direction, the stochastic gradient and Hessian along certain direction also easily follow. For instance, for function $$F_{I_h}(X)\triangleq\frac{1}{b_h}\sum_{i\in I_h}f_{i}(X),$$ the second term in the formula for $\bU^s_t$ (see Step 9 in Algorithm~\ref{alg:svr_cubic}) can be calculated as
\begin{align}
	\nabla^2 F_{I_h}(X)[U]&=[\frac{\nu+p}{2b_h}\sum_{i\in I_h}\frac{-(\ba_i\ba_i^T)\otimes(\ba_i\ba_i^T)}{(\nu+\ba_i^T X\ba_i)^2}+\frac{1}{2}X^{-1}\otimes X^{-1}]\cdot \text{vec}(U),\\
	\text{Hess} F_{I_h}(X)[U]&=X\text{sym}(\nabla^2 F_{I_h}(X)[U])X+\text{sym}(U\nabla F_{I_h}(X)X).
\end{align}
While computing the Euclidean gradient and Hessian (along certain direction) using \eqref{numerical:gradient} and~\eqref{numerical:tensor} requires processing $N$ data points, transforming them to their Riemannian counterparts is relatively simple, in the sense that their computation is independent of $N$. Therefore, at the beginning of each epoch, the tensor inside the square bracket in~\eqref{numerical:tensor} is computed and stored. In the following within-epoch iterations, to update $\bU^s_t$ (Step 9 in Algorithm~\ref{alg:svr_cubic}), only the second and third terms need to be updated which can be performed efficiently as the batch size is small compared to $N$.

\subsection{Linear Classifier Over the Sphere Manifold}
In this example, we consider a classification problem based on $N$ training examples $\{\ba_i,b_i\}_{i=1}^N$ where $\ba_i\in\mR^m$ and $b_i\in\{-1,1\}$ for all $i\in[N]$. We aim to estimate the model parameter $\bx$ for a linear classifier $f(\ba)=\bx^{\top}\ba$ such that 
%given a new feature $\hat{\ba}$, it predicts $\hat{b}=1$ if $f(\hat{\ba})>0$ and $\hat{b}=-1$, otherwise. 
it minimizes a smooth nonconvex loss function~\cite{zhao2010convex,li2003loss}
\begin{equation}\label{numerical:sphere}
	\cL(\bx; \{(\ba_i, b_i)\}_{i=1}^N)=\sum_{i=1}^N(1-\frac{1}{1+e^{-b_i\cdot \bx^\top \ba_i}})^2,
\end{equation}
over the Sphere manifold, $\{\bx \in\mR^m: \bx^\top \bx=1\}$~\cite{absil2009optimization}. The Euclidean gradient and Hessian of $\cL$ are 
\begin{equation*}
	\nabla \cL(\bx)=\sum_{i=1}^N -\frac{e^{-2b_i(\bx^\top \ba_i)}}{(1+e^{-b_i(\bx^\top\ba_i)})^3}\ba_i,
\end{equation*}
\begin{equation*}
	\nabla^2 \cL(\bx)=\sum_{i=1}^N \frac{(2-e^{-b_i\bx^\top\ba_i})b_i^2e^{-2b_i(\bx^\top \ba_i)}}{(1+e^{-b_i(\bx^\top\ba_i)})^4}\ba_i\ba_i^\top. 
\end{equation*}
The Riemannian gradient and Hessian of $\cL$ along $U$ (see Proposition~5.3.2 in \cite{absil2009optimization}, \cite{manopt}) are
\begin{align}
	\text{grad}\cL(\bx)&=P_{\bx}(\nabla \cL(\bx)),\\
	\text{Hess}\cL(\bx)[\bu]&=P_{\bx}(\nabla^2 \cL(\bx)[\bu])-(\bx^\top\nabla \cL(\bx))\bu,
\end{align}
where the tangent space projection is $P_{\bx}(\by)\triangleq\by-(\bx^\top \by)\bx$. The stochastic gradient and Hessian easily follows by summing the corresponding terms over the minibatch.
%\begin{equation*}
%	\text{grad} \cL(\beta)=\sum_{i=1}^N -\frac{e^{-2y_i(\beta^\top \bx_i)}}{(1+e^{-y_i(\beta^\top\bx_i)})^3}\bx_i+\frac{e^{-2y_i(\beta^\top \bx_i)}(\beta^\top\bx_i)}{(1+e^{-y_i(\beta^\top\bx_i)})^3}\beta,
%\end{equation*}
%\begin{align*}
%	\text{Hess} \cL(\beta)[\eta]&=\sum_{i=1}^N \frac{(2-e^{-y_i\beta^\top\bx_i})y_i^2e^{-2y_i(\beta^\top \bx_i)}}{(1+e^{-y_i(\beta^\top\bx_i)})^4}\bx_i\bx_i^\top\eta\\
%	&-\frac{(2-e^{-y_i\beta^\top\bx_i})y_i^2e^{-2y_i(\beta^\top \bx_i)}(\beta^\top\bx_i\bx_i^\top\eta)}{(1+e^{-y_i(\beta^\top\bx_i)})^4}\beta+\frac{e^{-2y_i(\beta^\top \bx_i)}(\beta^\top\bx_i)}{(1+e^{-y_i(\beta^\top\bx_i)})^3}\eta.
%\end{align*}

The first example above on estimating the inverse scale matrix of the multivariate t-distribution over symmetric positive definite (SPD) satisfies Assumptions~\ref{assumption:embedded}-\ref{assumption:curvature}, and its objective function satisfies Assumptions~\ref{assumption:smooth}-\ref{assumption:bounded} if the minimum eigenvalue of the matrices is bounded away from zero. The second example on estimating the parameter of the linear classifier over sphere manifold satisfies all of the assumptions, i.e., the sphere manifold satisfies Assumptions~\ref{assumption:embedded}-\ref{assumption:curvature} and Assumptions~\ref{assumption:smooth}-\ref{assumption:bounded} follows from continuous differentiability of the objective function and compactness of the sphere manifold.

\subsection{Numerical Results} \label{sec:numerical}
\paragraph{Data Simulation.}
The first numerical study is to estimate the inverse scale (covariance) matrix of the multivariate Student's t-distribution (see problem~\eqref{numerical:student_distribution}). Data is simulated from a multivariate t-distribution with three degrees of freedom and randomly generate scale matrix $\Sigma_{\text{true}}\in\cS^d_{++}$ with $d=10$. $N=10^4$ samples are generated from the  underlying distribution which are then added with the Gaussian noise $\epsilon$ sampled from $\mathcal{N}(0,\tau^2\bI_d)$ with $\tau^2$ equal to 0.1, 1, 5, and 10.

The second numerical study is to estimate the parameter of a linear classifier over the Sphere manifold (see problem~\eqref{numerical:sphere}). To simulate the data, the true parameter $\bx_{\text{true}}$ is first generated from $\mathcal{N}(0,\bI_d)$ which is then normalized to belong to the Sphere manifold. Next, $\ba_i\in \mR^{d\times1},\ i=1,...,N$ are randomly generated from the uniform distribution where $d=20$ and $N=10^5$. The corresponding label $b_i$ to $\ba_i$ is set to $1$ if $\bx^\top \ba_i+\epsilon_i>0$, where $\epsilon_i\sim\mathcal{N}(0,\tau^2)$, and $-1$, otherwise, where $\tau^2$ is chosen to be 0.02, 0.1, 1, and 3.

The proposed R-SVRC algorithm is run 15 times in each numerical study. The shaded plots discussed in the Results below provide percentile information based on these replicates.

\paragraph{Number of calls to the stochastic oracle.} For the R-SVRC method, at the beginning of each epoch, the SO is called $N$ times. However, within each epoch, each iteration makes $(b_g+b_h)$ calls to SO. In the deterministic CRC, ARC and RTR methods, each iteration makes $N$ calls to SO. The number of SO calls and the CPU runtime are the two performance measures we have used to compare the proposed method with the other second-order methods.

\paragraph{Parameters and subproblem solver.}
The g-smoothness and H-smoothness assumptions is standard in nonasympototic analysis in Riemannian optimization - see, e.g.,~\cite{absil2004trust,absil2009optimization,boumal2019global,boumal2020introduction}. However, obtaining the g-smoothness and H-smoothness constants is not trivial and we defer it to future studies. In the following, we numerically analyze the effect of different parameters on the performance of Algorithm~\ref{alg:svr_cubic}, i.e., epoch size $T$, cubic regularization constant $\sigma$, batchsize $b_g$ and $b_h$. The cubic subproblem (Step 9 of the Algorithm~\ref{alg:svr_cubic}) is solved using the conjugated gradient method using the Manopt solver~\cite{manopt}.

To estimate the inverse scale matrix of the multivariate t-distribution over the symmetric positive definite manifold, the default optimization parameter setting for Algorithm~\ref{alg:svr_cubic} is $\sigma=0.01$, $b_g=b_h=500$ and $T=5$. To estimate the parameter of the linear classifier over Sphere manifold, the default optimization parameter setting for Algorithm~\ref{alg:svr_cubic} is $\sigma=0.1$, $b_g=b_h=5000$ and $T=5$.

\paragraph{Results.}
Figure~\ref{fig:different_noise_level} shows the performance of the R-SVRC algorithm for different levels of noise $\epsilon$ added to the simulated data (see data simulation above). The top two plots in Figure~\ref{fig:different_noise_level} show the proposed algorithm successfully approach a second-order stationary point in all scenarios. As the output of Algorithm~\ref{alg:svr_cubic} is to be sampled uniformly at random for $s\in[S]$ and $t\in[T]$, we also plot the averaged $\mu(\bx^k)$ sequence (over iterations) in the bottom two plots. These averaged sequences show $\mE(\mu(\bx^k))$ decreases with a sublinear rate which is consistent with the first main theorem.

\begin{figure}[H]
    \centering
        \caption{Effect of the added noise to the simulated data on the performance of the proposed R-SVRC algorithm over 15 replicates. \textbf{(Left)} Estimating inverse scale matrix of multivariate t-distribution over SPD manifold. \textbf{(Right)} Estimating parameter of the linear classifier over Sphere manifold. 
    %\textbf{(Top)} Decrease of $\mu(\bx^k)$ with respect to the iterate $k$. \textbf{(Below)} Decrease of averaged $\mu(\bx^k)$ over the iterate $k$.
    }
    \label{fig:different_noise_level} 
    \begin{subfigure}[b]{0.49\textwidth}
        \centering
        \includegraphics[width=1.1\textwidth]{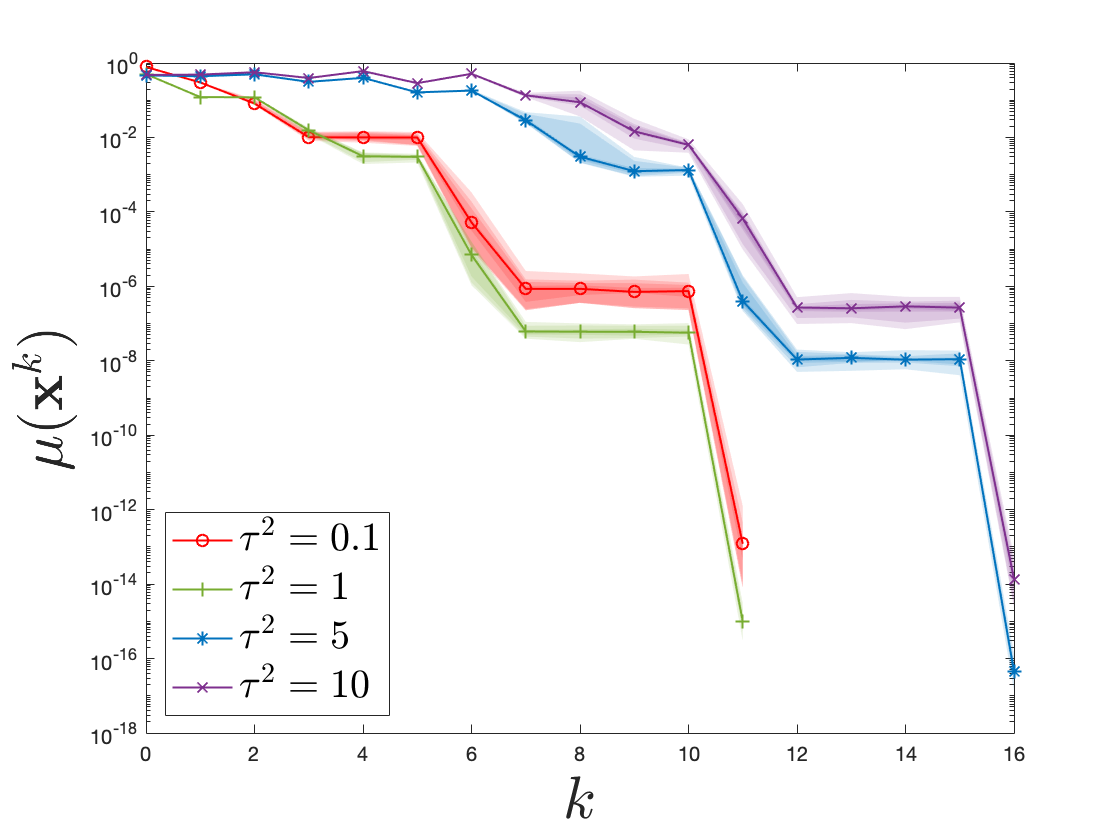}
        \caption*{}
         %\label{fig:1}
    \end{subfigure}
     %\hfill
    \begin{subfigure}[b]{0.49\textwidth}
        \centering
        \includegraphics[width=1.1\textwidth]{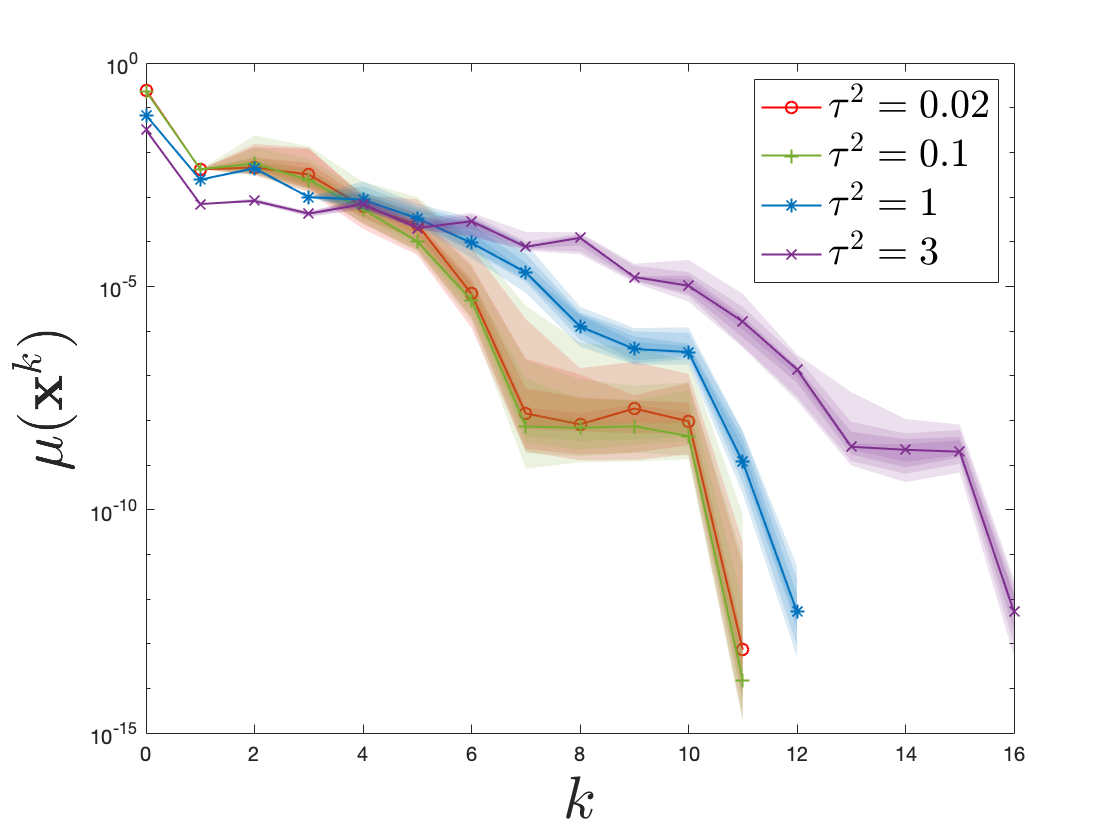}
         \caption*{}
         %\label{fig:4}          
    \end{subfigure}
    %\hfill
    \begin{subfigure}[b]{0.49\textwidth}
        \centering
        \includegraphics[width=1.1\textwidth]{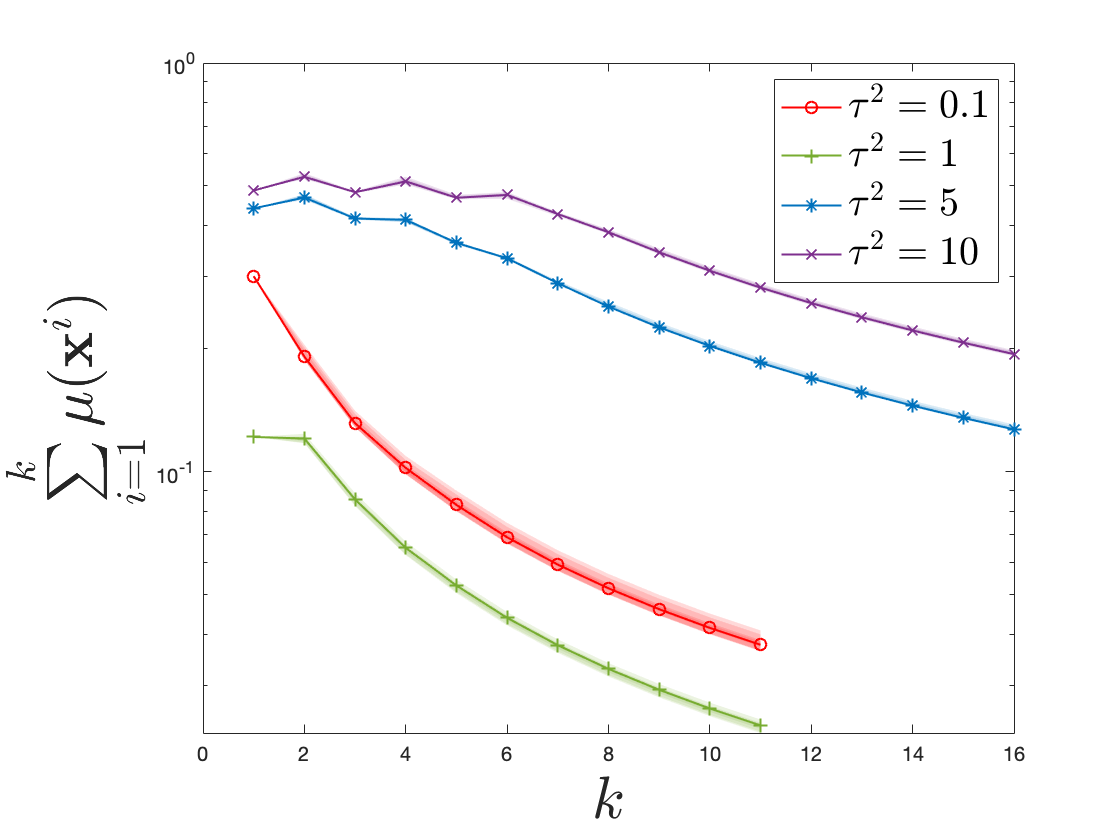}
         \caption*{}
         %\label{fig:2}
     \end{subfigure}
     %\hfill
    \begin{subfigure}[b]{0.49\textwidth}
        \centering
        \includegraphics[width=1.1\textwidth]{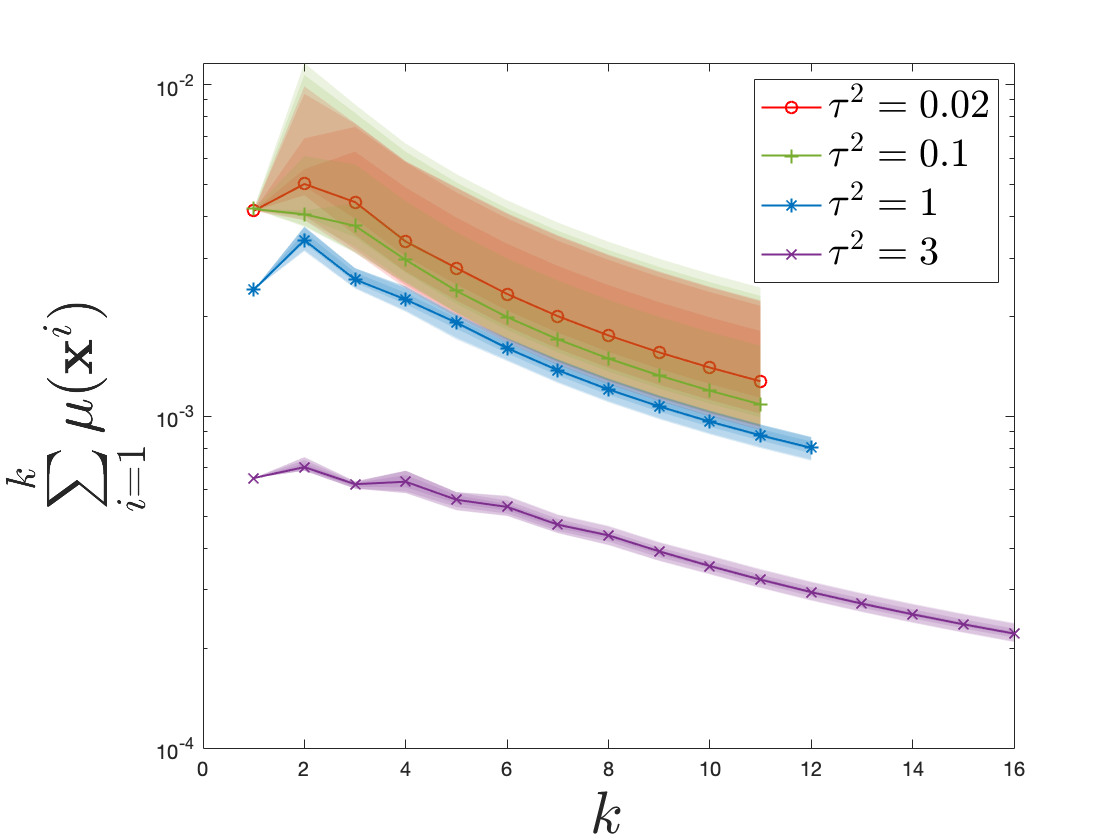}
         \caption*{}
         %\label{fig:3}
    \end{subfigure}
\end{figure} 

\begin{figure}[H]
    \centering
        \caption{Performance of of the proposed R-SVRC algorithm for different optimization parameter settings over 15 replicates. \textbf{(Left)} Estimating inverse scale matrix of multivariate t-distribution over SPD manifold. \textbf{(Right)} Estimating parameter of the linear classifier over Sphere manifold.}
        \label{fig:different_parameter}
    \begin{subfigure}[b]{0.49\textwidth}
        \centering \includegraphics[width=1.0\textwidth]{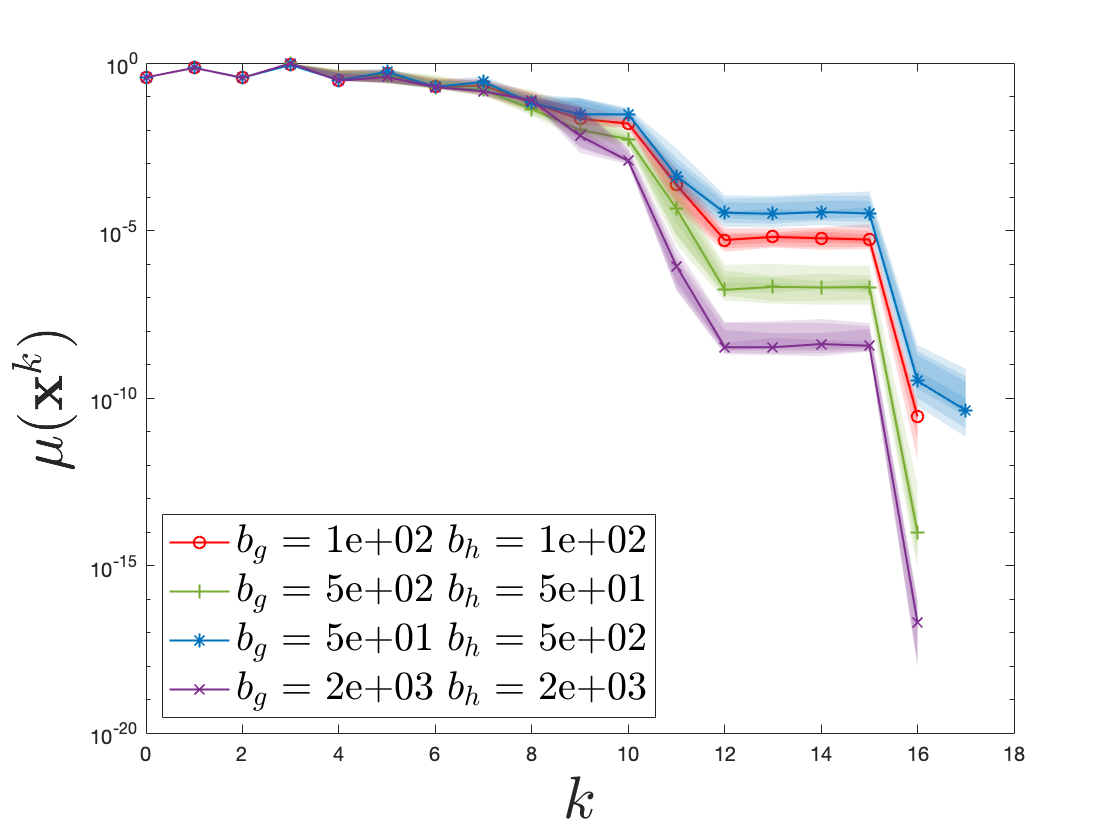}
        \caption*{}
         %\label{fig:1}
    \end{subfigure}
     %\hfill
    \begin{subfigure}[b]{0.49\textwidth}
        \centering
        \includegraphics[width=1.0\textwidth]{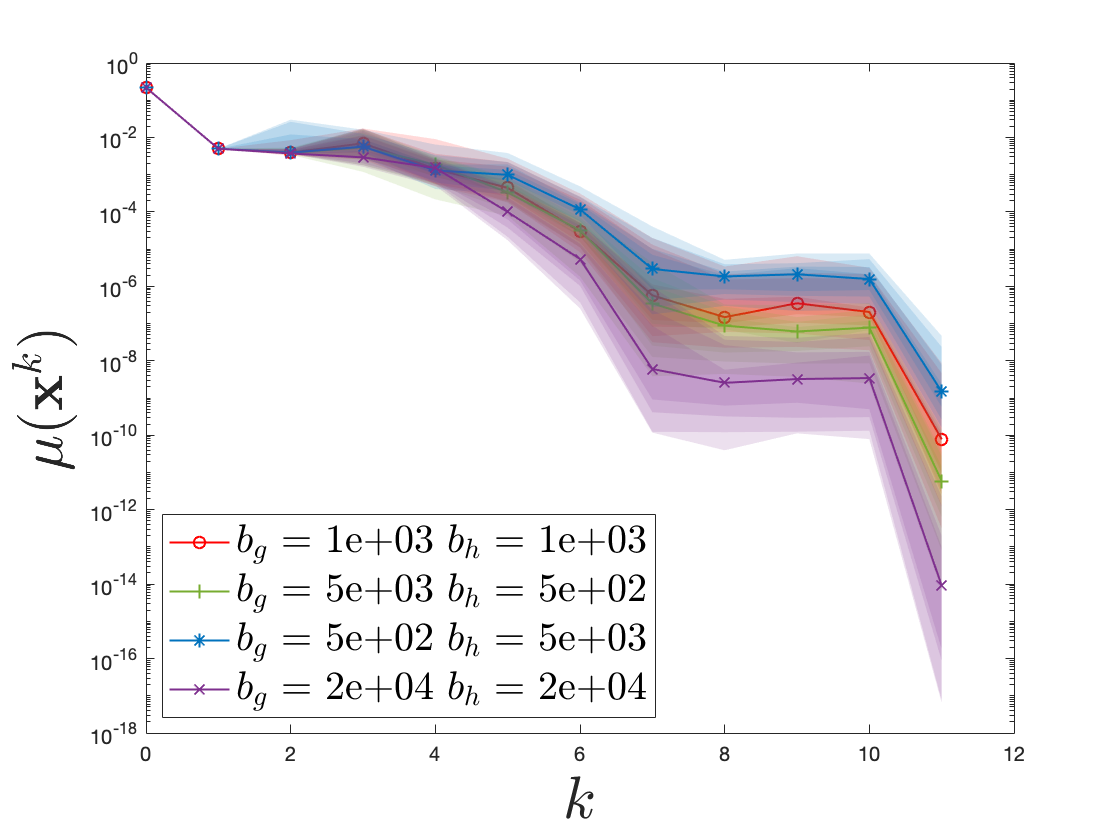}
         \caption*{}
         %\label{fig:4}          
    \end{subfigure}
    %\hfill
    \begin{subfigure}[b]{0.49\textwidth}
        \centering
        \includegraphics[width=1.0\textwidth]{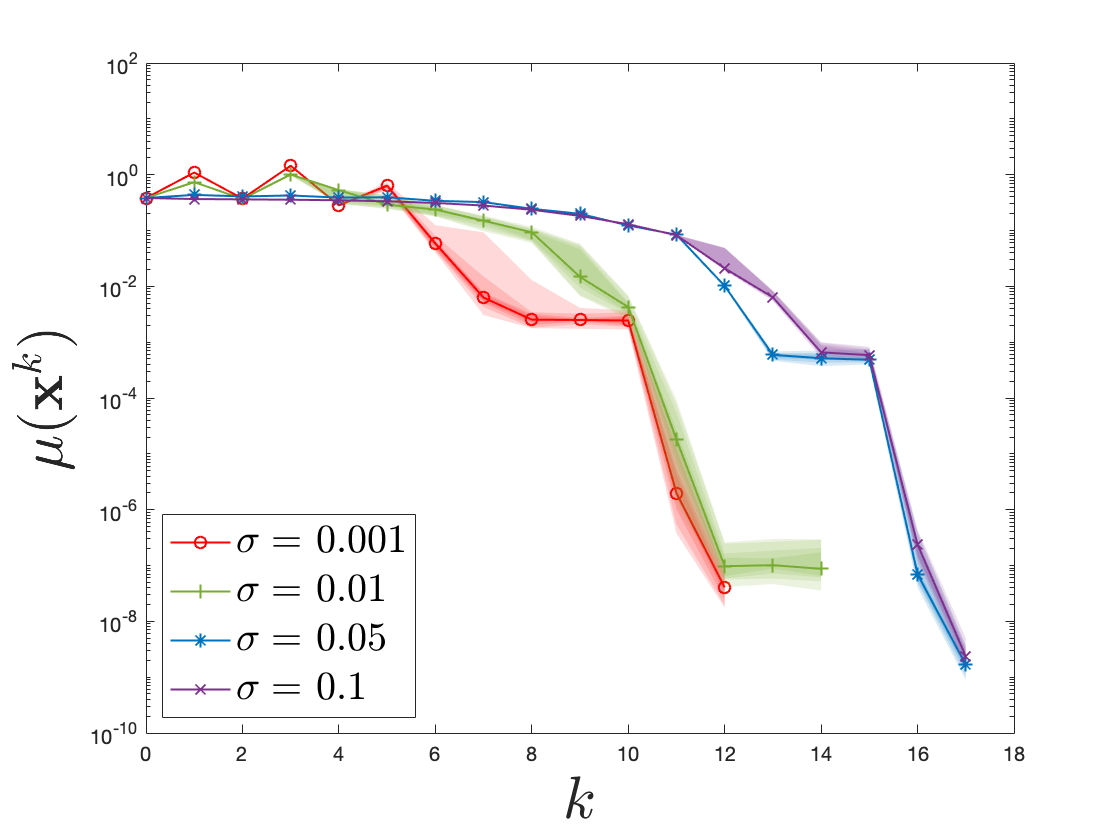}
         \caption*{}
         %\label{fig:2}
     \end{subfigure}
     %\hfill
    \begin{subfigure}[b]{0.49\textwidth}
        \centering
        \includegraphics[width=1.0\textwidth]{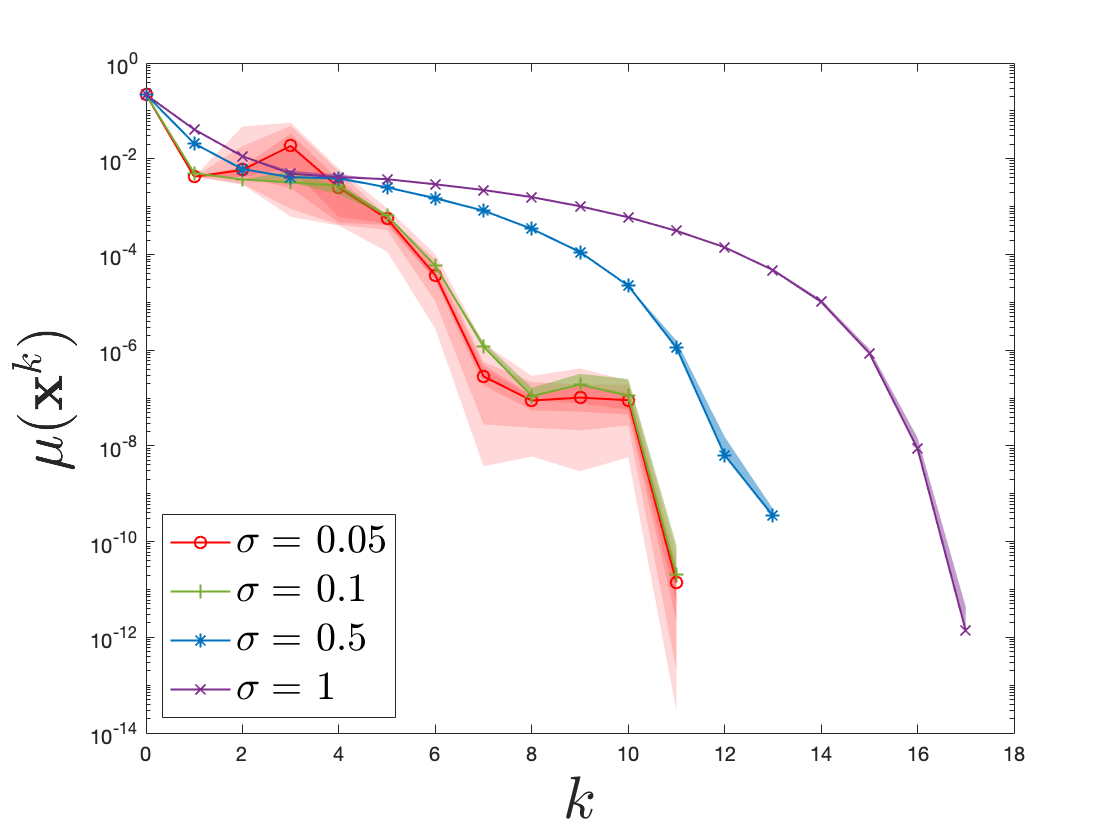}
         \caption*{}
         %\label{fig:3}
    \end{subfigure}
     %\hfill
    \begin{subfigure}[b]{0.49\textwidth}
        \centering
        \includegraphics[width=1.0\textwidth]{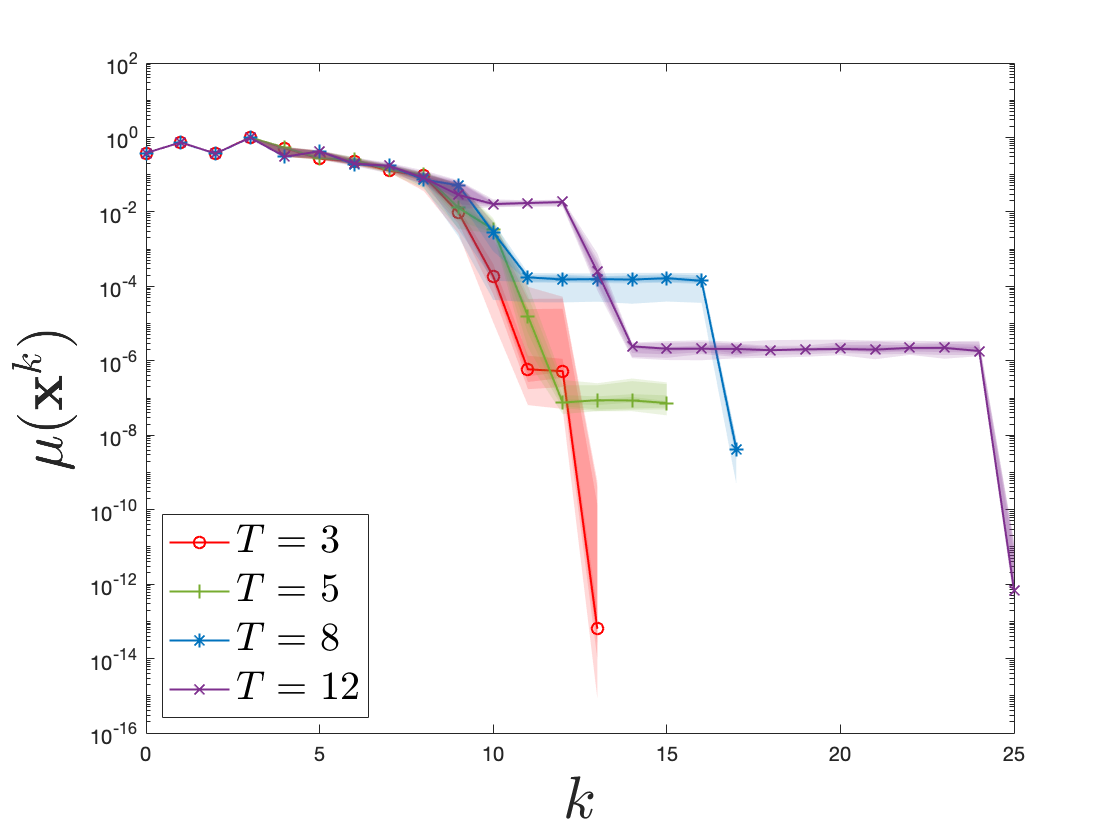}
         \caption*{}
         %\label{fig:3}
    \end{subfigure}
       % \hfill
    \begin{subfigure}[b]{0.49\textwidth}
        \centering
        \includegraphics[width=1.0\textwidth]{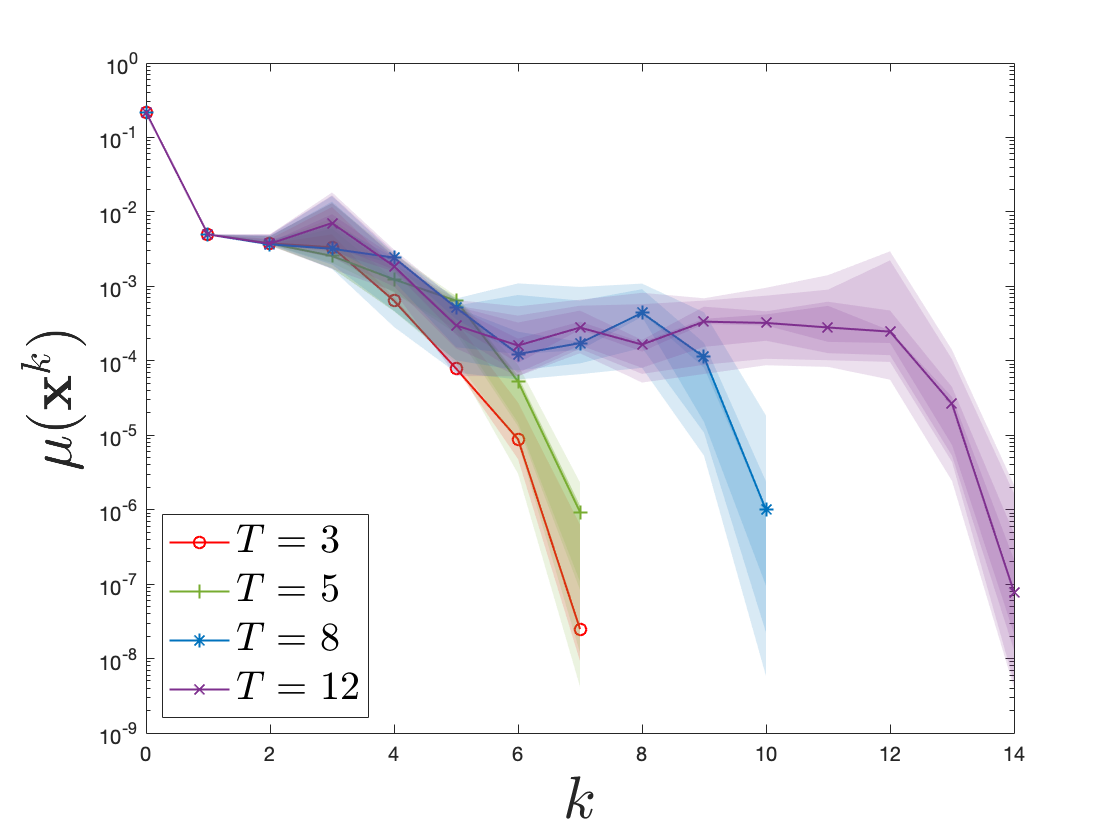}
         \caption*{}
         %\label{fig:4}          
    \end{subfigure}
\end{figure}

Figure~\ref{fig:different_parameter} illustrates the performance of the R-SVRC algorithm for both numerical studies for different settings of the optimization parameters. Left and right columns corresponds to the first and second numerical studies, respectively. Most of the plots show a superlinear rate of convergence to a second-order stationary point using the last iterate as the output of the algorithm. The first row shows that smaller batch sizes result in slower convergence with early oscillation around the plateau. Specifically, the top three lines have ascending values of $b_g$ while descending values of $b_h$ which implies that the effect of $b_g$ is more significant than that of $b_h$. The second row shows that bigger values of $\sigma$ can provide smaller objective values but with slower convergence rate. Furthermore, larger $\sigma$ values tends to produce a smaller $\norm{\bh}$ based on the subproblem \eqref{m_cubic} which leads to more stable and smooth sequences shown in the plots. The third row shows that bigger values of $T$, i.e., less frequent full gradient and Hessian calculations, result in slower rate of convergence for a fixed number of iterations which is also intuitive.

In Figure~\ref{fig:compare_different_methods}, we compare the proposed R-SVRC method with the other three benchmark methods, Riemannian adaptive cubic regularization method (ARC), Riemannian trust region method (RTR) and crude Riemannian cubic regularization method (CRC)~\cite{agarwal2020adaptive,boumal2015riemannian,zhang2018cubic} over the number stochastic oracle calls (see Definition~\ref{def:SO}) and also cpu time. Results show faster decrease by the R-SVRC method compared to the other benchmark methods. %over these two measures. %Note that in the below two plots in Figure~\ref{fig:compare_different_methods}, we report the 15 replications without constructing the percentile shape.

\begin{figure}
    \centering
    \begin{subfigure}[b]{0.49\textwidth}
        \centering
        \includegraphics[width=1.0\textwidth]{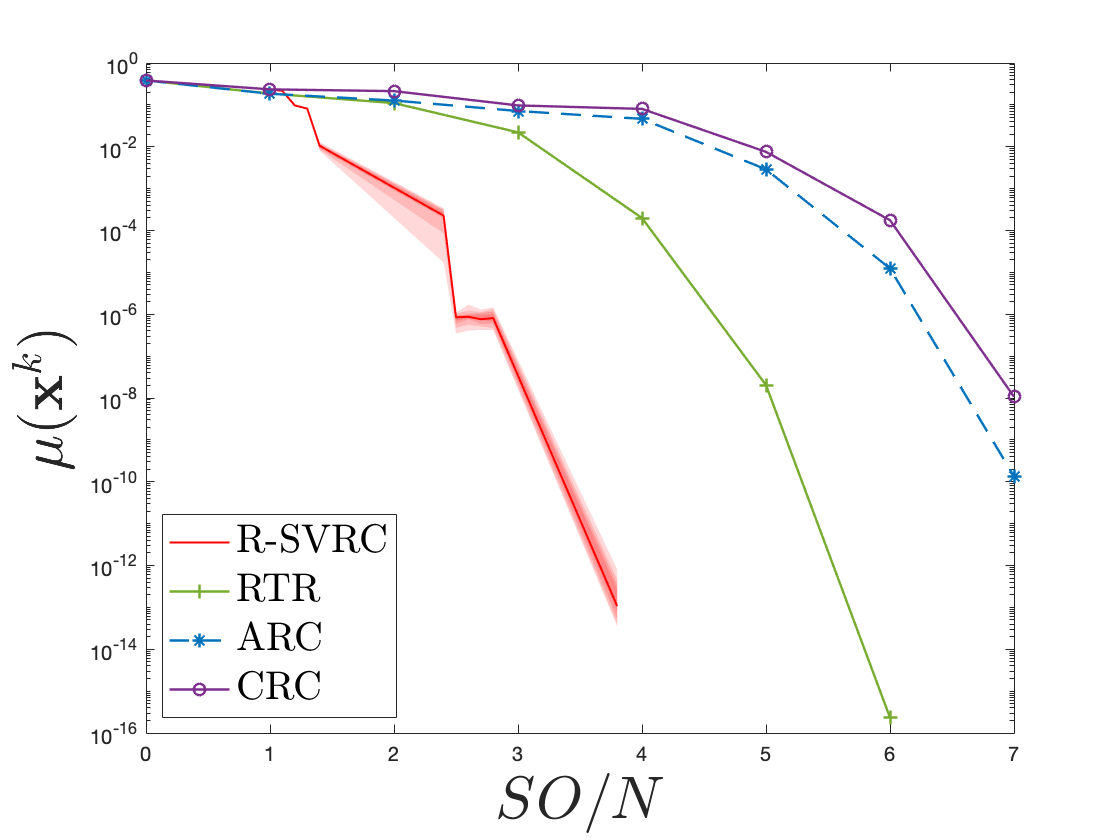}
        \caption*{}
         %\label{fig:1}
    \end{subfigure}
     %\hfill
    \begin{subfigure}[b]{0.49\textwidth}
        \centering
        \includegraphics[width=1.0\textwidth]{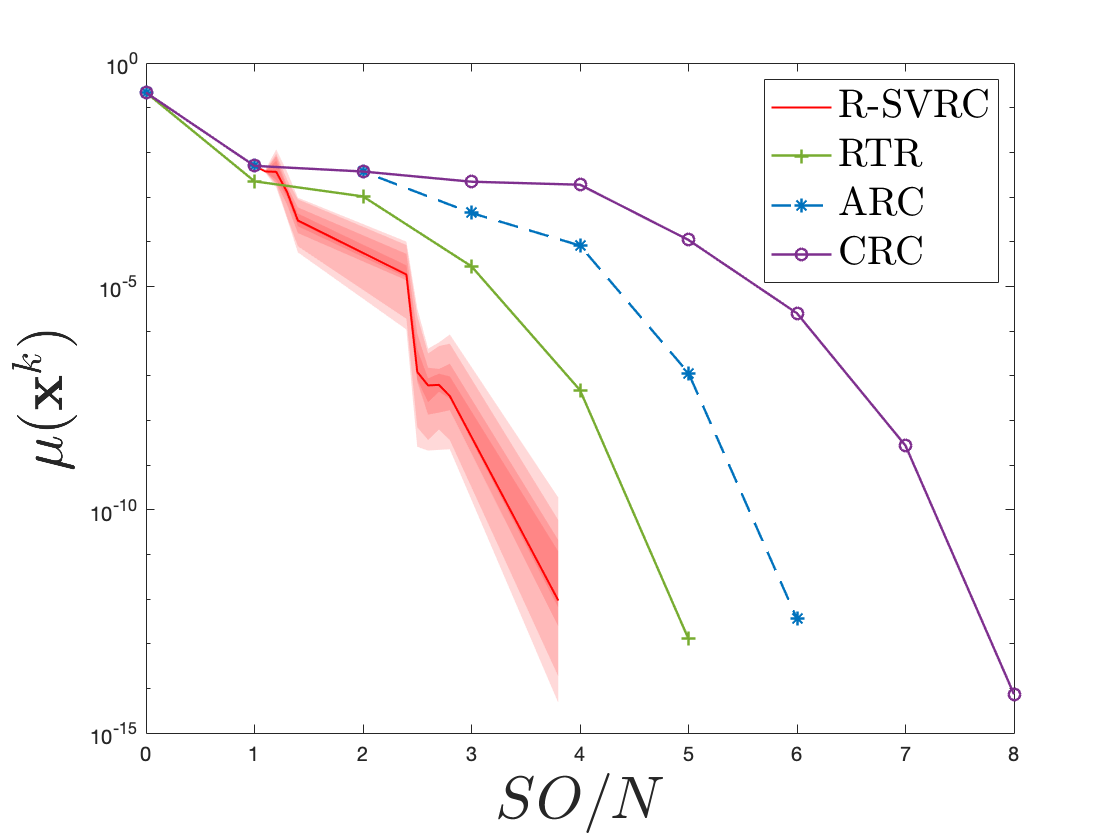}
         \caption*{}
         %\label{fig:4}          
    \end{subfigure}
    %\hfill
    \begin{subfigure}[b]{0.49\textwidth}
        \centering
        \includegraphics[width=1.0\textwidth]{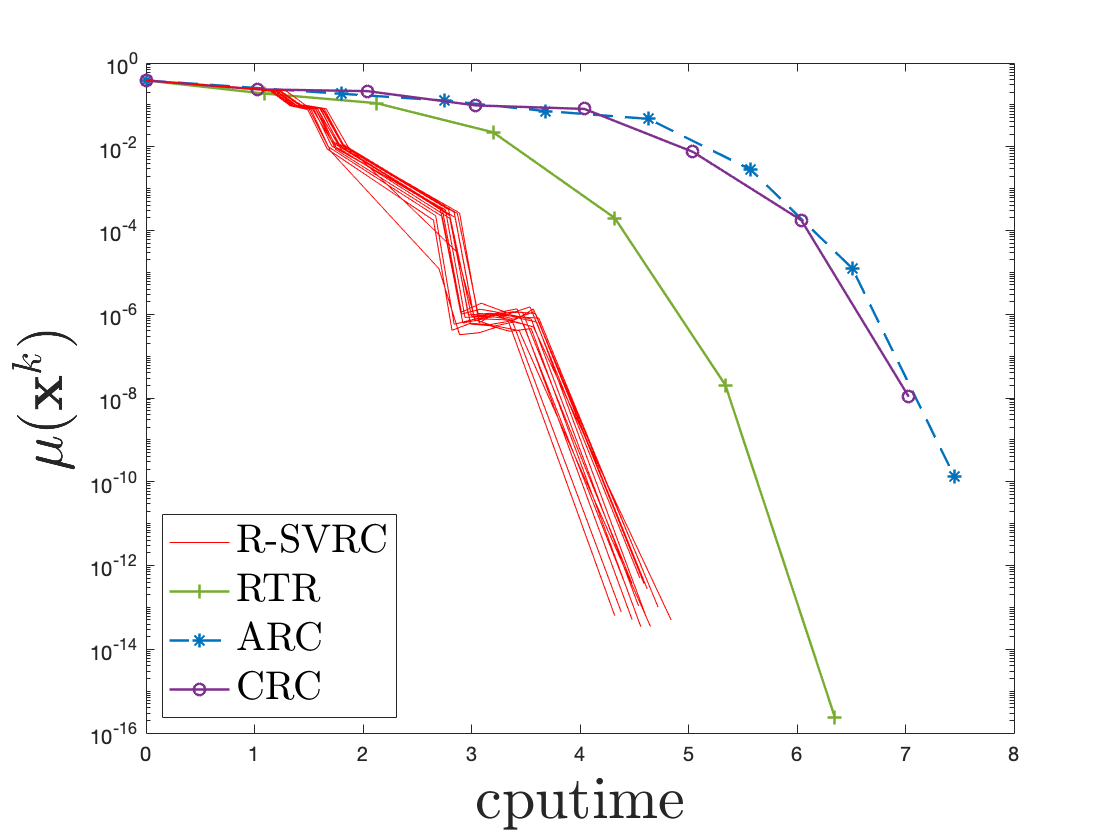}
         \caption*{}
         %\label{fig:2}
     \end{subfigure}
     %\hfill
    \begin{subfigure}[b]{0.49\textwidth}
        \centering
        \includegraphics[width=1.0\textwidth]{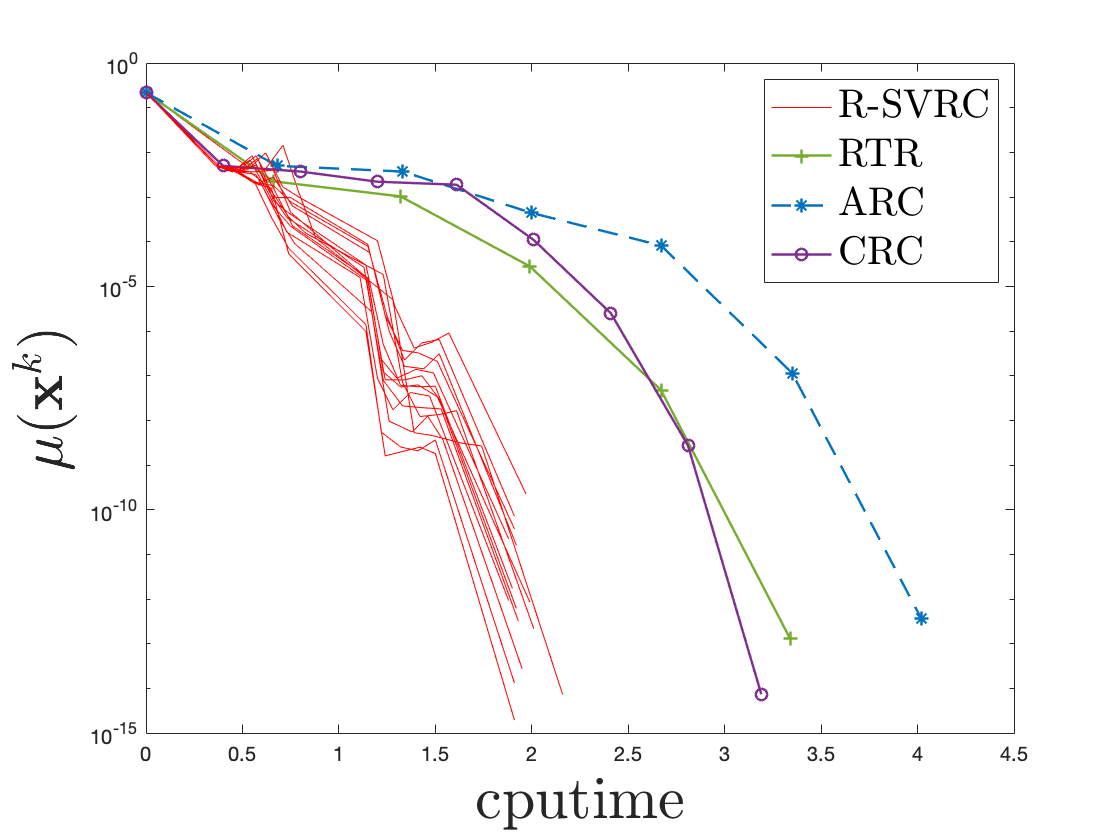}
         \caption*{}
         %\label{fig:3}
    \end{subfigure}   
    \caption{Performance of the proposed R-SVRC algorithm compared to the three benchmark methods. \textbf{(Left)} Estimating inverse scale matrix of multivariate t-distribution over SPD manifold. \textbf{(Right)} Estimating parameter of the linear classifier over Sphere manifold.  
    %\textbf{(Top)} Comparison from the perspective of number of SO oracles. \textbf{(Below)} Comparison from the perspective of CPU running time.
    }
    \label{fig:compare_different_methods}
\end{figure}

Finally, Figure~\ref{fig:threeDim} visualizes the optimization path obtained by the R-SVRC algorithm over the Sphere manifold. The generated iterates converge to the optimal solution. 

\begin{figure}
        \centering
        \includegraphics[width=0.7\textwidth]{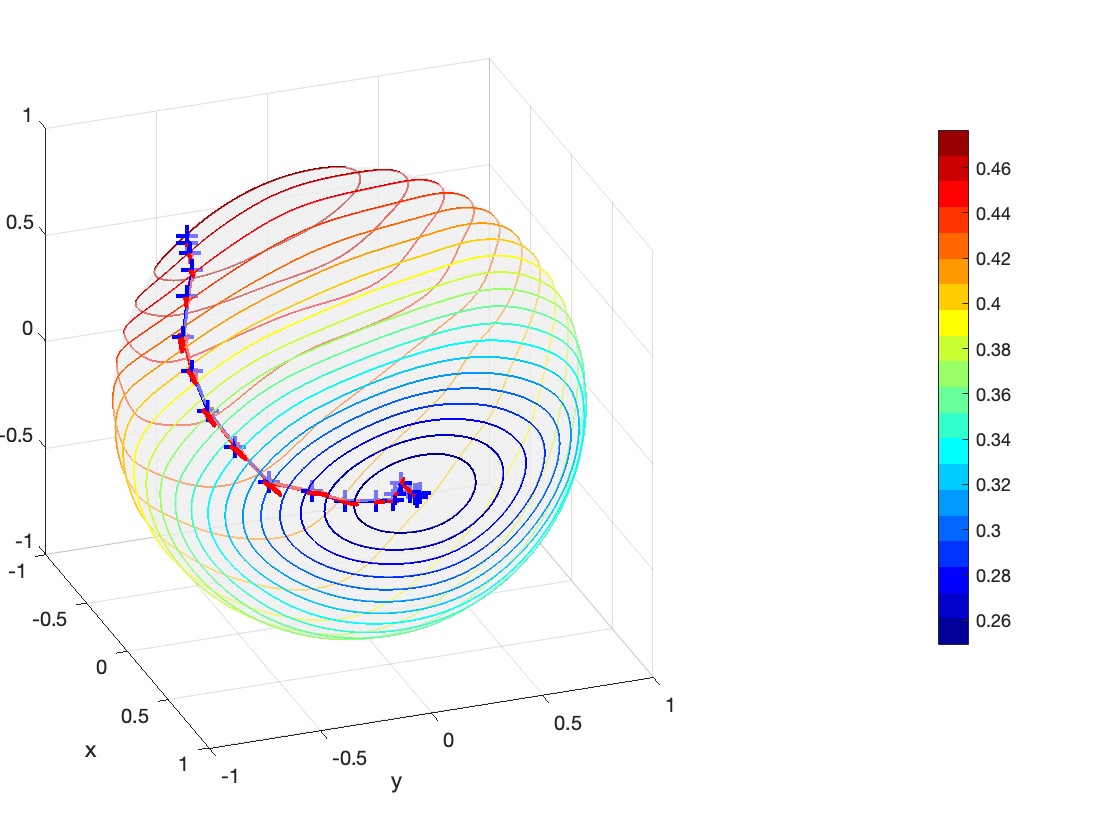}
    \caption{The path of the R-SVRC algorithm iterates over the 2-dimensional instance of the Sphere manifold in the second numerical study.}
    \label{fig:threeDim}
\end{figure}

\vspace{-0.2cm}
\section{Conclusions}\label{sec:conclude}
We developed the Riemannian stochastic variance-reduced cubic-regularized Newton method (R-SVRC) for optimization over Riemannian manifolds embedded in a Euclidean space. The proposed double-loop algorithm requires information on the full gradient and Hessian at the beginning of each epoch (outer loop) but updates the gradient and Hessian within each epoch  in a stochastic variance-reduced fashion. Each iteration requires solving a cubic-regularized Newton subproblem. Iteration complexity of the proposed algorithm to find a second-order stationary points is established which matches the worst-case complexity bounds in the Euclidean setting. Furthermore, a version of the algorithm which only requires an inexact solution to the cubic regularized Newton subproblem is proposed which has the same complexity bound as the exact case. Finally, the performance of the proposed algorithm is evaluated over two numerical studies with symmetric positive definite and sphere manifolds.

\bibliographystyle{authordate1} 
\bibliography{refs}

\end{document}